\documentclass[11pt]{article}
\usepackage[english]{babel}
\usepackage{amsmath, amssymb, amsthm}
\usepackage{amscd}
\usepackage{float}

\newcommand{\ack}{\section*{Acknowledgments}}

\newcommand{\be}{\begin{equation}}
\newcommand{\ee}{\end{equation}}
\newcommand{\bea}{\begin{eqnarray}}
\newcommand{\eea}{\end{eqnarray}}
\newcommand{\bean}{\begin{eqnarray*}}
\newcommand{\eean}{\end{eqnarray*}}

\newcommand{\xa}{\alpha}
\newcommand{\xb}{\beta}
\newcommand{\xg}{\gamma}

\newcommand{\xd}{\delta}
\newcommand{\xD}{\Delta}
\newcommand{\xe}{\varepsilon}
\newcommand{\xz}{\zeta}

\newcommand{\xl}{\lambda}
\newcommand{\xL}{\Lambda}

\newcommand{\xm}{\mu}
\newcommand{\xn}{\nu}

\newcommand{\xr}{\rho}
\newcommand{\xs}{\sigma}

\newcommand{\xf}{\phi}
\newcommand{\xF}{\Phi}

\newcommand{\xo}{\omega}
\newcommand{\xO}{\Omega}

\newcommand{\rd}{\mathrm{d}}
\newtheorem{theorem}{Theorem}[section]
\newtheorem{lemma}[theorem]{Lemma}
\newtheorem{prop}[theorem]{Proposition}

\newtheorem{defin}[theorem]{Definition}
\newtheorem{remark}[theorem]{Remark}

\begin{document}
\title{Boundary Singularities on a Wedge-like Domain of a Semilinear Elliptic Equation}
\author{\textbf{ Konstantinos T. Gkikas}\\
Centro de Modelamiento
Matem\'{a}tico\\
Universidad de Chile,
Av. Blanco Encalada 2120 Piso 7
Santiago, Chile.\\
Email: kgkikas@dim.uchile.cl }

\maketitle
\begin{abstract}
Let $n\geq2$ and $ \xO\subset \mathbb{R}^{n+1}$ be a Lipschitz wedge- like domain . We construct positive weak
solutions of the problem
$$\xD u + u^p = 0 \quad\hbox{in }\, \xO,$$
 which vanish in a suitable trace sense on $\partial\xO,$ but which are
singular at prescribed ``edge" of $\xO$ if $p$ is equal or slightly above a certain exponent $p_0>1$ which depends on $\xO.$ Moreover, in the case which $\xO$ is unbounded, the solutions have fast decay at infinity.

\smallskip

\end{abstract}

\section{Introduction}

Let $\xO$ be a bounded domain in $\mathbb{R}^n,$ $n\geq2$ with smooth boundary $\partial\xO.$ A model of nonlinear
elliptic boundary value problem is the classical Lane-Emden-Fowler equation,
\begin{equation}
\begin{cases}
-\xD u&= |u|^p\qquad\mathrm{in}\phantom{\partial}\qquad\xO,\\
      u&>\phantom{|}0\phantom{|^p}\qquad\mathrm{in}\phantom{\partial}\qquad\xO,\\
     u&=\phantom{|}0\phantom{|^p}\qquad\mathrm{in}\qquad\partial\xO,
     \end{cases}
     \label{eksi}
\end{equation}
where $p>1.$ Following Brezis and Turner \cite{3} and Quittner and Souplet \cite{9}, we will say that a positive
function $u$ is a very weak solution of problem (\ref{eksi}), if $u$ and $\mathrm{dist}(x,\partial\xO)u^p\;\in\;L^1(\xO),$ and
$$\int_{\xO}u\xD v+|u|^pv\rd x=0,\qquad\forall v\in C^2(\overline{\xO}),\;\mathrm{with}\;v=0\;\mathrm{on}\;\partial\xO.$$
From the results in \cite{3,9}, it follows that if $p$ satisfies the constraint
\be
1<p<\frac{n+1}{n-1},\label{crit1}
\ee
then $u\in C^2(\overline{\xO}),$ i.e. $u$ is a classical solution of problem \eqref{eksi}.

It is well known that, if $1<p<\frac{n+2}{n-2},$ one can use Sobolev's embedding and standard variational techniques to prove the existence of a positive very weak solution of problem (\ref{eksi}). However, if $\frac{n+1}{n-1}<p<\frac{n+2}{n-2},$ this very weak solution may not be bounded. A result in the understanding of very weak solutions was achieved by Souplet \cite{10}. He constructed an example of a positive function $a\in L^\infty(\xO)$ such that problem (\ref{eksi}), with $u^p$ replaced by $a(x)u^p$ for $p > \frac{n+1}{n-1},$ has a
very weak solution which is unbounded, developing a point singularity on the boundary. This shows that the exponent $p=\frac{n+1}{n-1}$
is truly a critical exponent. Let us mention that the study
of the behavior near an isolated boundary singularity of any positive solution of (1.1) when the
exponent $p\geq \frac{n+1}{n-1}$ was achieved by Bidaut-V\'{e}ron-Ponce-V\'{e}ron in \cite{2}. Finally, del Pino-Musso-Pacard \cite{pino} showed the existence of $\xe>0$ such that for any $p\in[\frac{n+1}{n-1},\frac{n+1}{n-1}+\xe)$ an unbounded, positive, very weak solution of (\ref{eksi}) exists which blows up at a prescribed point of $\partial\xO.$ For the respective problem with interior singularity see for example \cite{4,5,6,8}.

Let us give some definitions for convenience to the reader.
Let $n\geq2$ and $(r, \theta) \in[0,\infty)\times\mathbb{S}^{n-1}$ be the spherical-coordinates of $x\in\mathbb{R}^n$ abbreviated
by $x = (r, \theta).$ Given an open Lipschitz spherical cap $\xo\subsetneq\mathbb{S}^{n-1}$ let
$$C_{\xo}=\{x=(r,\theta):\;r>0,\;\theta\in\xo\},$$
be the corresponding infinite cone. The set
$$C_\xo^R=C_\xo\cap B_R(0)\subset\mathbb{R}^n$$
is called a conical piece with spherical cap $\xo$ and radius $R.$

A bounded Lipschitz domain $\xO\subset C_\xo$ is called a domain with a conical boundary piece if there exists a
conical piece $C_\xo^R$ such that $\xO\cap B_R(0)=C_\xo^R.$

We denote by $\xl$ and $\xf_1(\theta)$ to be respectively  the first eigenvalue and the corresponding eigenfunction of the problem
\begin{align}
\begin{cases}
-\xD_{\mathbb{S}^{n-1}}u&=\xl u\qquad \mathrm{in}\; \xo\\
\phantom{-\xD_{\mathbb{S}^{n-1}}} u&=0\phantom{u}\qquad \mathrm{on}\; \partial\xo,
\end{cases}\label{eigenin}
\end{align}
with $\int_{\xo}\xf_1^2\rd S_x=1.$

Finally, we define the exponent
\be
p^*=\frac{n+\xg}{n+\xg-2};\qquad\;\hbox{with }\qquad\xg=\frac{2-n}{2}+\sqrt{\left(\frac{n-2}{2}\right)^2+\xl},\label{gamma}
\ee
and note that $p^*$ depends on $\xo.$

In the same spirit as above, McKennab-W. Reichel \cite{mr} generalized the results of  Souplet \cite{10} to domain with conical boundary piece, and they showed that the exponent $p^*$ is a truly critical exponent, in the sense that, if $1<p<p^*,$ then every very weak solution of problem (\ref{eksi}) is bounded (see also \cite{1}). Finally, Hor\'{a}k-McKennab-Reichel \cite{HMR} considered a bounded Lipschitz domain $\xO$ with a conical boundary piece of spherical cap  $\xo\subset \mathbb{S}^{n-1},$ at
$0\in\partial\xO,$ and they proved the existence of $\xe>0$ such that for any $p\in(p^*,p^*+\xe)$ an unbounded, positive, very weak solution of (\ref{eksi}) exists which blows up at $0\in\partial\xO.$

Let us  consider the following problem
\begin{align}
\begin{cases}
\xD_x u+u^p&=0,\qquad\mathrm{in}\;C_{\xo}\\
u\phantom{p}&>0,\qquad\mathrm{in}\;C_{\xo}\\
u\phantom{p}&=0,\qquad\mathrm{on}\;\partial C_{\xo}\setminus\{0\}.\end{cases}\label{pro1in}
\end{align}

The authors in \cite{HMR} proved that problem (\ref{pro1in}) admits a positive solution of the form $w(\theta)=|x|^{-\frac{2}{p-1}}\xf_p(\theta),$  where $\xf_p$ solves the problem
\begin{align}
\nonumber
\xD_{\mathbb{S}^{n-1}} \xf-\frac{2}{p-1}\left(-\frac{2}{p-1}+n-2\right)\xf+\xf^p&=0,\qquad\mathrm{in}\;\xo\\
\xf\phantom{p}&=0,\qquad\mathrm{on}\;\partial\xo,\label{pro2in}
\end{align}
for any $p\in(p^*,\infty)$ if $n=2,\;3$ and any $p\in(p^*,\frac{n+1}{n-3})$ if $n\geq4.$ But this solution does not have fast decay at infinity.

We note here that if $\xo=\mathbb{S}^{n-1}_+,$ then $\xg=1,$ thus the critical exponent  $p^*=\frac{n+1}{n-1}$ and $C_{\xo}=\mathbb{R}^n_+.$  In \cite{pino}, del Pino-Musso-Pacard constructed a solution of problem (\ref{pro1in}) in $\mathbb{R}^n_+$ with fast decay. More precisely they showed that there exists $\xe>0$ such that for any $p\in(\frac{n+1}{n-1},\frac{n+1}{n-1}+\xe)$  problem (\ref{pro1in}) in $\mathbb{R}^n_+$ admits a solution $u\in C^2(\mathbb{R}^n_+)$ satisfying
$$u(x)\approx|x|^{-\frac{2}{p-1}}\xf_p(\theta),\qquad\mathrm{as}\;|x|\rightarrow0$$
 and
$$u(x)\approx|x|^{-(n-1)}\xf_1(\theta),\qquad\mathrm{as}\;|x|\rightarrow\infty.$$

The first result of this work is the construction of a singular solution at $0$ with fast decay at infinity, for problem (\ref{pro1in}). In particular we prove
\begin{theorem}
There exists a number $p(n,\xl)>p^*,$ such that for any $$p\in(p^*,p(n,\xl)),$$ there exists a solution $u_1(x)$ to problem (\ref{pro1in}) such that
$$u_1(x)=|x|^{-\frac{2}{p-1}}\xf_p(\theta)(1+o(1))\qquad\mathrm{as}\;|x|\rightarrow0,$$
where $\xf_p$ solves (\ref{pro2in}), and
$$u_1(x)=|x|^{2-\xg-n}\xf_1(\theta)(1+o(1))\qquad\mathrm{as}\;|x|\rightarrow\infty,$$
where $\xg$ is defined in (\ref{gamma}).
In addition, we have the pointwise estimate
$$
|u_1(x)|\leq C |x|^{-\frac{2}{p-1}}||\xf_p||_{\mathcal{C}^2(\xo)}
,
$$
for some constant $C>0$ which does not depend on $p.$\label{mainprop}
\end{theorem}
To describe our main result let us introduce some new notations.

Let $x\in\mathbb{R}^n$ with $n\geq2$.  Given $\tau\in \mathbb{R}$, we let $\xo(\tau)\subsetneq\mathbb{S}^{n-1}$ to be the corresponding Lipschitz spherical cap.
We set $$r_{\xs(\tau)}=|x-\xs(\tau)|,$$
where $\xs:\mathbb{R}\rightarrow\mathbb{R}^n$ is a smooth curve such that
$$\sup_{\tau\in\mathbb{R}}\left\{|\xs(\tau)|+|\xs'(\tau)|+|\xs''(\tau)|\right\}<C<\infty.$$

Now, given $\tau$, we  let $(r_{\xs(\tau)}, \theta) \in[0,\infty)\times\mathbb{S}^{n-1}$ to be the spherical-coordinates of $x\in\mathbb{R}^n$ centered at $\xs(\tau)$ abbreviated
by $x = (r_{\xs(\tau)}, \theta).$ We define
$$\widetilde{C}_{\xo(\tau)}=\{x=(r_{\xs(\tau)},\theta):\;r_{\xs(\tau)}>0,\;\theta\in\xo(\tau)\}\subset\mathbb{R}^n$$
and we  set
$$\xO_{\tau_1,\tau_2}=\{(\tau,x)\in(\tau_1,\tau_2)\times\mathbb{R}^n: x\in \widetilde{C}_{\xo(\tau)}\}\subset\mathbb{R}^{n+1},$$

$$\xO_{\tau_1,\tau_2}^R=\xO_{\tau_1,\tau_2}\cap \{(\tau,x)\in(\tau_1,\tau_2)\times\mathbb{R}^{n}:\;x\in B_R(\xs(\tau))\}\subset\mathbb{R}^{n+1},$$ and
$$S_{\tau_1,\tau_2}=\{(\tau,x)\in[\tau_1,\tau_2]\times\mathbb{R}^n: r_{\xs(\tau)}=0\}.$$
Finally we define $\xl^*=\inf\limits_{\tau\in\mathbb{R}}\xl(\tau)$ and $\xg^*=\inf\limits_{\tau\in\mathbb{R}}\xg(\tau).$

\

In this work we assume that $\xo(\tau)$ depends smoothly on $\tau,$ i.e. $\xl(\tau)$ is a smooth bounded function with respect $\tau$ with bounded derivatives.
We also assume that $\inf\limits_{\tau\in\mathbb{R}}\xl(\tau)>0.$ Finally, we suppose that there exists $\xe>0,$ such that for any $p\in\left(\sup\limits_{\tau\in\mathbb{R}}p^*(\tau),\sup\limits_{\tau\in\mathbb{R}}p^*(\tau)+\xe\right),$ there exists a solution $u_1(\tau,x)$ of theorem \ref{mainprop}. That means, $\text{osc}_{\tau\in\mathbb{R}}\xl(\tau)$ is small enough.
\begin{theorem}
Let $\xe>0$ be small enough. Then there exists a number $p_0>\sup\limits_{\tau\in \mathbb{R}}p^*$ such that, given $p\in(\sup\limits_{\tau\in \mathbb{R}}p^*,p_0),$ and $\frac{2}{p-1}\leq-\rho< n+\xg^*-2,$ the following  problem
\begin{align}
\begin{cases}
-\xD u&=u^p\qquad\mathrm{in}\qquad \xO_{-\infty,\infty},\\ \nonumber
u&>0\phantom{^p}\qquad\mathrm{in}\qquad \xO_{-\infty,\infty}\\ \nonumber
u&=0\phantom{^p}\qquad\mathrm{on}\qquad  \partial\xO_{-\infty,\infty}\setminus S_{-\infty,\infty}
\end{cases}
\end{align}
possesses  very weak solutions $u$.
In addition we have that
$$u(\tau,x)\approx u_1\left(\tau,\frac{x-\xs(\tau)}{\xe}\right)\quad\hbox{as }\;r_{\xs(\tau)}\rightarrow0,$$
where $u_1$ is  in theorem \ref{mainprop}. And
$$u(\tau,x)\leq C\,r_{\xs(\tau)}^{\rho}\quad\hbox{as }\;r_{\xs(\tau)}\rightarrow\infty.$$\label{th1.2}
\end{theorem}
Our third and final result of this paper is the following
\begin{theorem}
Let $\xa>0$ be small enough and $\xO\subset\mathbb{R}^{n+1}$ be a bounded Lipschitz domain such that
$$\xO\cap\xO_{\tau_1-\xa,\tau_2+\xa}^R=\xO_{\tau_1-\xa,\tau_2+\xa}^R\subset\mathbb{R}^{n+1}.$$
There exists a number $p_0>\sup\limits_{\tau\in \mathbb{R}}p^*$ such that, given $p\in(\sup\limits_{\tau\in \mathbb{R}}p^*,p_0),$  there exist very weak solutions $u$ to the problem
\begin{align}
\begin{cases}
-\xD u&=u^p,\qquad\mathrm{in}\qquad \xO,\\ \nonumber
u&>0\phantom{^p},\qquad\mathrm{in}\qquad \xO\\ \nonumber
u&=0\phantom{^p},\qquad\mathrm{on}\qquad  \partial\xO\setminus S_{\tau_1-\xa,\tau_2+\xa}.
\end{cases}
\end{align}
Moreover, $\forall (\tau,x)\in\xO_{\tau_1-\frac{\xa}{4},\tau_2+\frac{\xa}{4}}^R$
$$u(\tau,x)\approx u_1\left(\tau,\frac{x-\xs(\tau)}{\xe}\right)\qquad\hbox{as }\;r_{\xs(\tau)}\rightarrow0.$$
\label{th1.3}
\end{theorem}

The paper is organized as follows. In section \ref{cone} we prove theorem \ref{mainprop}. In subsection \ref{u1}, we prove some regularity  results with respect $\tau,$ for the function $u_1(\tau,x)$ in theorem \ref{mainprop}. Section \ref{theorems} will be devoted to the proofs of theorems \ref{th1.2} and \ref{th1.3}.

\setcounter{equation}{0}
\section{The eigenvalue problem on spherical caps.}

Let $n\geq2,\;\tau\in\mathbb{R},$ and $\xo(\tau)\subsetneq\mathbb{S}^{n-1}$ be the corresponding open Lipschitz spherical cap. We denote by $\xl(\tau)$ and $\xf_1(\tau,\theta)$ to be respectively the first eigenvalue and eigenfunction of the eigenvalue  problem
\begin{align}
\begin{cases}
-\xD_{\mathbb{S}^{n-1}}u&=\xl(\tau) u,\qquad \mathrm{in}\; \xo(\tau)\\
\phantom{-\xD_{\mathbb{S}^{n-1}}} u&=0\phantom{u},\qquad \mathrm{on}\; \partial\xo,
\end{cases}\label{eigen}
\end{align}
with $\int_{\xo(\tau)}\xf_1^2dS_x=1.$

We assume that $\xo(\tau)$ depends smoothly on $\tau,$ i.e. $\xl(\tau)$ is a smooth bounded function with respect $\tau$ with bounded derivatives.
We  also assume that $\inf\limits_{\tau\in\mathbb{R}}\xl(\tau)>0.$

Now note that, without loss of generality, we can set $\theta_1=\cos t,$ with $0<t<\xb(\tau),$ where $\xb(\tau)$ is a smooth function with bounded derivatives satisfying
\begin{equation*}
\begin{cases}
&0<\inf\limits_{\tau\in\mathbb{R}}\xb(\tau)<\sup\limits_{\tau\in\mathbb{R}}\xb(\tau)<2\pi\qquad \hbox{for } \quad n=2\\
& \hbox{and } \\
& 0<\inf\limits_{\tau\in\mathbb{R}}\xb(\tau)<\sup\limits_{\tau\in\mathbb{R}}\xb(\tau)<\pi \qquad \hbox{for } \quad n\geq3.
\end{cases}
\end{equation*}

Then problem (\ref{eigen}) is equivalent to the following one
\begin{align}
\begin{cases}
-\sin^{2-n} t\frac{d}{dt}\left(\sin^{n-2} t\frac{d\xf_1}{dt}\right)&=\xl\xf_1\quad\mathrm{in}\;\;(0,\xb(\tau)).\\[3mm]
\xf_1(\xb(\tau))&=0\\[3mm]
\frac{d\xf_1}{dt}(0)&=0,
\end{cases}\label{eigenode*}
\end{align}
with
$$C(n)\int_0^{\xb(\tau)}\sin^{n-2}(t)|u|^2\rd t=\int_{\xo}|\xf_1|^2\rd S=1.$$
We note here that, for $n=2$ in problem (\ref{eigenode*}), we may have $\xf_1(0)=0$ instead of $\frac{d\xf_1}{dt}(0)=0.$

We have the following lemma
\begin{lemma}
Let $\xf_1(\tau,\theta)$ be the first eigenfunction of the following eigenvalue problem
\begin{align}
\nonumber
-\xD_{\mathbb{S}^{n-1}}u&=\xl u,\qquad \mathrm{in}\; \xo(\tau)\\
\phantom{-\xD_{\mathbb{S}^{n-1}}} u&=0\phantom{u},\qquad \mathrm{on}\; \partial\xo(\tau),\label{eigenz}
\end{align}
with $\int_{\xo(\tau)}\xf_1^2dS=1.$ Then there exists a positive constant $C$ such that
\be
\sup_{\tau\in\mathbb{R}}\left|\left||\xf_1|+\left|\frac{\partial\xf_1}{\partial \tau}\right|+\left|\frac{\partial^2\xf_1}{\partial \tau^2}\right|\right|\right|_{L^\infty(\xo(\tau))}<C.
\ee\label{f1}
\end{lemma}
We postpone the proof of this  lemma to the appendix.

\setcounter{equation}{0}
\section{Positive singular solution in the Cone}\label{cone}
We keep the assumptions and notations of the previous section, and we consider the cone
$$C_{\xo(\tau)}=\{(r,\theta):\;r>0,\;\theta\in\xo(\tau)\},$$
where $r=|x|$ and $\theta=\frac{x}{|x|}.$
We define the critical exponent
$$p^*(\tau)=\frac{n+\xg(\tau)}{n+\xg(\tau)-2}\quad \hbox{with }\quad \xg(\tau)=\frac{2-n}{2}+\sqrt{\left(\frac{n-2}{2}\right)^2+\xl(\tau)}.$$
We consider the problem
\begin{align}
\begin{cases}
\xD_x u+u^p&=0,\qquad\mathrm{in}\;C_{\xo(\tau)}\\
u\phantom{p}&>0,\qquad\mathrm{in}\;C_{\xo(\tau)}\\
u\phantom{p}&=0,\qquad\mathrm{on}\;\partial C_{\xo(\tau)}\setminus\{0\}.
\end{cases}\label{pro1}
\end{align}
If we set $w=|x|^{-\frac{2}{p-1}}\xf(\theta),$ we arrive at the problem
\begin{align}
\begin{cases}
\xD_{\mathbb{S}^{n-1}} \xf-\frac{2}{p-1}\left(-\frac{2}{p-1}+n-2\right)\xf+\xf^p&=0,\qquad\mathrm{in}\;\xo(\tau)\\
\xf\phantom{p}&=0,\qquad\mathrm{on}\;\partial\xo(\tau).
\end{cases}\label{pro2}
\end{align}
By lemma 9 in \cite{HMR}, problem (\ref{pro2}) has a positive solution $\xf_p\in H_1(\xo(\tau))\cap L^\infty(\xo(\tau))$ for any $p\in(p^*,\infty)$ if $n=2$ or $3$ and for any $p\in(p^*(\tau),\frac{n+1}{n-3})$ if $n\geq4.$  Also as $p\downarrow p^*(\tau)$ then $-\frac{2}{p-1}\left(-\frac{2}{p-1}+n-2\right)\uparrow\xl(\tau)$ and
$$\xf_p=\left(\frac{\xl-\frac{2}{p-1}\left(-\frac{2}{p-1}+n-2\right)}{c_p}\right)^\frac{1}{p-1}(\xf_1+o(1)),$$
where $c_p=\int_{\xo(\tau)}\xf_1^{p+1}d\theta.$

In addition, for the same range on $p,$ by theorem 10 in \cite{HMR}, the function $$w_p(\tau,r,\theta)=r^{-\frac{2}{p-1}}\xf_p(\tau,\theta)$$ is a positive solution of (\ref{pro1}).

\

In the rest of this section, for convenience, we omit dependence on the parameter $\tau$ writing $\xl=\xl(\tau),$ $\xf_1(\theta)=\xf_1(\tau,\theta)$ and so on.

Let $p\in(p^*,\frac{n+2}{n-2})$, we look for solutions of (\ref{pro1}) of the form
\be
u_1(x)=|x|^{-\frac{2}{p-1}}\xf(-\log|x|,\theta),\label{solu}
\ee
where $\theta=\frac{x}{|x|},$ so that the equation $\xD u+u^p=0$ reads in terms of the function $\xf$ defined for $t\in\mathbb{R}$ and $\theta\in \xo,$ as
\be
\partial^2_t\xf+A\xf_t-\xe\xf+\left(\xD_{\mathbb{S}^{n-1}}\xf+\xl\xf\right) +\xf^p=0,\label{pro3}
\ee
where $t=-\log r,$ $A=-\left(n-2\frac{p+1}{p-1}\right)$ and $\xe=\xl+\frac{2}{p-1}(n-\frac{2p}{p-1}).$

Let $\xm=\int_\xo\xf_1^{p+1}\rd\theta,$ we define $a_\infty$ by
$$\xm a_\infty^{p-1}=\xe.$$
We look for a positive function $a$ which is a solution of
\be
a''(t)+ Aa'(t)-\xe a(t)+\xm a^p(t)=0,\label{2.23}
\ee
which converges to $0$ as $t$ tends to $-\infty$ and converges to $a_\infty$ as $t$ tends to $+\infty.$ Observe that,
when $p\in(p^*,\frac{n+2}{n-2}),$ the coefficients $A$ and $\xe$ are positive and, therefore, in this range, classical
ODE techniques yield the existence of $a,$ a positive heteroclinic solution of (\ref{2.23}) tending to $0$
at $-\infty$ and tending to $a_\infty$ at $+\infty$.

\

Observe that since the equation \eqref{2.23} is  autonomous,  the function $a$ is
not unique and $a$ can be normalized so that $a(0) = \frac{1}{2}a_\infty$.  For more  informations about  the
function $a$, we refer the reader to   lemmas 2.3, 2.4, 2.5 and appendix in \cite{pino}.

\begin{prop}
Let $0\leq p_0<\infty$ and $\xe$ be small enough, then there exists a unique operator
$$G_{p_0}:\;a^{p_0}L^\infty(\mathbb{R}\times\xo)\mapsto a^{p_0}L^\infty(\mathbb{R}\times\xo),$$
such that for any $a^{-p_0}g\in L^\infty(\mathbb{R}\times\xo),$ the function $u=G_{p_0}(g)$ is the unique solution of
$$L_pu=\left(\partial^2_t+A\partial_t-\xe+\left(\xD_{\mathbb{S}^{n-1}}+\xl\right) +p\xf_0^{p-1}\right)u=g;\;\;\;\xf_0=a(t)\xf_1(\theta),$$
with zero Dirichlet boundary data.

Furthermore,
\bea\left|\left|d^{-1}a^{-p_0}(t)\psi\right|\right|_{L^\infty(\mathbb{R}\times\xo)}\leq \frac{C}{\xe}\left|\left|a^{-p_0}(t)g\right|\right|_{L^\infty(\mathbb{R}\times\xo)}.\label{eq:phigp0}\eea
If in addition $g(t,\cdot)$ is $L^2-$orthogonal to $\xf_1$ for a.e. $t ,$ then we have
$$\left|\left|d^{-1}a^{-p_0}(t)\psi\right|\right|_{L^\infty(\mathbb{R}\times\xo)}\leq C\left|\left|a^{-p_0}(t)g\right|\right|_{L^\infty(\mathbb{R}\times\xo)}$$
where $d:\mathbb{\xo}\rightarrow(0,\infty)$ denotes the distance function to $\partial\xo.$\label{propest}
\end{prop}
\begin{proof}
The proof follows the  same lines as in lemma 2.6 in \cite{pino}, so we will  only focus on the differences.
We first define $\xf_*$ to be  the positive solution of
\bea
\begin{cases}
\xD_{\mathbb{S}^{n-1}}\xf_*+\xl\xf_*+\xd(\xd-n-2\xg+2)\xf_*=&-1\qquad\mathrm{in}\;\xo\\
\xf_*=&\phantom{-}0\qquad\mathrm{on}\;\partial\xo
\end{cases}
\label{phi*}
\eea
see the proof of lemma 2.6 in \cite{pino} with obvious modifications.
Using the function $(t, \theta)\rightarrow e^{-\xd t}\xf_*(\theta)$
as a barrier, as done in the paper \cite{pino}, we can  show that, given any function $g$ such that $a^{-p_0}g\in L^\infty(\mathbb{R}\times\xo)$ and
given $t_1 < -1 < 1 < t_2,$ we can solve the equation
$$
L_pu = g
$$
in $(t_1, t_2)\times \xo$ with $0$ boundary conditions.

\

To prove the estimate \eqref{eq:phigp0}, we argue by contradiction,
assuming  that
$$||a^{-p_0}\psi_i||_{L^\infty}=1$$
and
$$\lim_{i\rightarrow\infty}||a^{-p_0}f_i||=0$$
we get a contradiction using similar argument as  in  lemma 2.6 in \cite{pino}.
The rest of the proof is the same as in  lemma 2.6 in \cite{pino} with obvious modifications so we omit it here.
\end{proof}
\begin{proof}[Proof of theorem \ref{mainprop}]
We look for  a solution to problem (\ref{pro3}) of the form
$$\xf=a(t)\xf_1(\theta)+\psi(t,\theta),$$
and we let $G_p$ to be the operator defined in proposition \ref{propest}. To conclude the proof, it is enough to find
a function $\psi$ solution of the fixed point problem
$$
\psi=-G_p(\mathcal{M}(\xf_0)+\mathcal{Q}(\psi)),
$$
where
\begin{eqnarray*}
\xf_0(t,\theta)&=&a(t)\xf_1(\theta), \\[3mm]
\mathcal{M}(\xf_0)&=&a^p\left(\xf_1^p-\xm\xf_1\right)\\[3mm]
\mathcal{Q}(\psi)&=&|\xf_0+\psi|^p-\xf^p_0-p\xf^{p-1}_0\psi.
\end{eqnarray*}
The rest of the proof is the same as in \cite{pino}. We recall here that $\psi<<a\xf_1.$
Also in \cite{pino}, they have proven that if $\xe$ is small enough then there exists $t_0$ such that for any $t\leq-\frac{t_0}{\xe},$
$$\frac{1}{2}e^{\xd^-t}\leq a(t)\leq e^{\xd^-t},$$ with  $\xd^-=\frac{1}{2}\left(\sqrt{A^2+4\xe}-A\right).$ And the result follows, since
$$\frac{1}{2}\left(\sqrt{A^2+4\xe}-A\right)+\frac{2}{p-1}=n+\xg-2.$$
\end{proof}
\begin{remark}
\label{remark}
\end{remark}
If $1<p_0<p$ is close enough to $p,$ we can apply a fix point argument like in the proof of theorem \ref{mainprop}, for the operator $G_{p_0}.$

In view of the proof of lemma \ref{f1}, $\xf_*=\xf_*(t,\cos (s\xb(\tau))).$

Thus if the function $g$ in proposition \ref{propest} is of the form $g=g(t,\cos (s\xb(\tau))),$ we have that the solution $u=G_{p_0}(g)$ is of the form
$u=u(t,\cos(s\xb(\tau))).$ Hence we obtain, that the solution $u_1$ in theorem \ref{mainprop} is of the form $$u_1=r^{-\frac{2}{p-1}}u_1(r,\cos (s\xb(\tau))).$$

\subsection{Regularity of the solution $u_1$ with respect $\tau$}\label{u1}

We first recall some definitions and known results, see the  book of Gilbarg and Trudinger \cite{gil} for the proofs.

Let
$$Lu=a^{i,j}(x)D_{i,j}u+b^i(x)D_iu+c(x)u=g(x),\qquad a^{i,j}=a^{j,i},$$
where the coefficients $a^{i,j}$, $b^i$, $c$  and the function $g$ are defined in an open bounded domain $\xO\subset\mathbb{R}^n$ and
$$a^{i,j}\xi_i\xi_j\leq\xm |\xi|^2;\quad \xm>0.$$

We assume that
$$||a^{i,j}||_{C^{2,a}},\;||b^{i}||_{C^{2,a}},\;||c||_{C^{2,a}}\leq\xL.$$
\begin{defin}
We say that a  bounded domain $\xO\subset\mathbb{R}^n$ and its boundary $\partial\xO$ are of class  $C^{k,a},\;0\leq a\leq1,$ if at each point $x\in\partial\xO$ there is a ball $B_r(x)$ and a one-to-one mapping $\psi$ from  $B_r(x)$ onto $D\subset\mathbb{R}^n$ such that:
$$\psi(B_r(x)\cap\xO)\subset \mathbb{R}^n_+,\; \psi(B_r(x)\cap\partial\xO)\subset\partial\mathbb{R}^n_+,\;\psi\in C^{k,a}(B_r(x))\;\mathrm{and}\;\psi^{-1}\in C^{k,a}(D).$$
A domain $\xO$ will be said to have a boundary portion $T\subset \partial\xO$ of class $C^{k,a},$ if at each
point $x\in T$ there is a ball $B_r(x)$ in which the above conditions are satisfied
and such that $B_r(x)\cap\partial\xO\subset T.$
\end{defin}
\begin{prop}{\textbf{(Lemma 6.18 in \cite{gil}})}. Let $0<a\leq1$ and $\xO$ be a domain with a $C^{2,a}$ boundary portion T, and let $\xf\in C^{2,a}(\overline{\xO}).$
Suppose that $u$ is a $C^2(\xO)\cap C_0(\overline{\xO})$ function satisfying
$Lu=g$ in $\xO$, $u=\xf$ on $T,$
where $g$ and the coefficients of the strictly elliptic operator $L$ belong to $C^a(\overline{\xO}).$ Then
$u\in C^{2,a}(\xO\cup T).$\label{gil1}
\end{prop}
\begin{prop} \textbf{(Corollary 6.7 in \cite{gil})}. Let $0<a\leq1$ and $\xO$ be a domain with a $C^{2,a}$ boundary portion T, and let $\xf\in C^{2,a}(\overline{\xO}).$
Suppose that $u$ is a $C^{2,a}(\xO\cup T)$ function satisfying $Lu=g$ in $\xO$, $u=\xf$ on $T.$ Then, if $x\in T$ and $B= B_\xr(x)$ is a ball with radius $\xr < \mathrm{dist} (x, \partial\xO - T),$ we have
$$||u||_{C^{2,a}(B\cap\xO)}\leq C(n,\xm,\xL,\xO\cap B_\xr(x))\left(||u||_{C(\xO)}+||\xf||_{C^{2,a}(\overline{\xO})}+||g||_{C^{a}(\xO)}\right).$$\label{gil2}

\end{prop}
We first prove the following result
\begin{lemma}
Let $\tau\in\mathbb{R}$ be fixed, $x\in\mathbb{R}^n,\;n\geq2,$ $g\in C^a(\overline{C_\xo}\setminus\{0\})$
and $u=G_p(g)$ be the operator in proposition \ref{propest}. Then
\bea
\nonumber
|\nabla_x u(\tau,x)|&\leq& C(n,p,\xl,C_{\xo(\tau)},g)\;|x|^{-1}\\
|D^2_xu(\tau,x)|&\leq& C(n,p,\xl,C_{\xo(\tau)},g)\;|x|^{-2}.\label{ass}
\eea\label{lemma11}
\end{lemma}
\begin{proof}
 First we note that $||u(\tau,\cdot)||_{L^\infty(C_\xo(\tau))}\leq C||g(t,\cdot)||_{L^\infty(C_\xo(\tau))}$ and $u$ is a solution of
\begin{align}
\begin{cases}
-\xD_x u+\frac{4}{p-1}\frac{x\cdot \nabla_x u}{|x|^2}+\frac{2}{p-1}\left(n-\frac{2}{p-1}-2\right)\frac{u}{|x|^2}
-p\frac{\xf_0^{p-1}u}{|x|^2}=-\frac{g}{|x|^2},\quad\mathrm{in}\;\; C_{\xo(\tau)}&\\[3mm]
u=0\qquad\mathrm{in}\;\; \partial C_{\xo(\tau)}\setminus\{0\}.&
\end{cases}
\end{align}
Set $R=|x|$,  consider the domain $$\xO_{R}=\{y\in C_\xo:\;\frac{R}{4}<|y|<4R\},$$
 and let  $y=\frac{x}{R}$ and define $v(y)=u(\tau,Ry).$ Then $y\in\xO_1$ and $v$ is a solution of
\begin{align}
\begin{cases}
&-\xD v+\frac{4}{p-1}\frac{y\cdot \nabla v}{|y|^2}+\frac{2}{p-1}\left(n-\frac{2}{p-1}-2\right)\frac{v}{|y|^2}
-p\frac{\xf_0^{p-1}v}{|y|^2}=-\frac{g}{|y|^2},\quad\mathrm{in}\;\; \xO_1\\[3mm]
&v=0\qquad\mathrm{in}\;\; T,
\end{cases}
\end{align}
where we have set
$$T=\partial\xO_1\setminus\{y\in C_\xo:\;|y|=\frac{1}{4}\;\mathrm{or}\;|y|=4\}.$$
Let $0<\xe<\frac{\xr}{4}$ be small enough, where $\xr$ is the defined  in proposition \ref{gil2} with $\xO=\xO_1.$  Let $y_0\in \partial\xO_1\setminus\{y\in C_\xo:\;|y|=\frac{1}{6}\;\mathrm{or}\;|y|=\frac{8}{3}\}$ then by propositions \ref{gil1} and \ref{gil2} we have
\begin{align}
\nonumber
||v||_{C^{2}(B_\xr(\psi_0)\cap\xO_\frac{2}{3})}&\leq C(n,\xm,\xL,\xO_1\cap B_\xr(y_0))||g||_{C^a(\overline{\xO_1})}
\end{align}
where in the last inequality we have used the estimate in proposition \ref{propest}.

We note here that $\xr$ depends only on $\xO_1$ and not on $y_0.$ Thus if we apply a covering argument and standard interior Schauder estimates we have
$$||v||_{C^{2}(\xO_\frac{1}{2})}\leq C\left(n,\xm,\xL,\xO_1,\xr\right)||g(x)||_{C^a(\overline{\xO_1})}.$$
Using the facts that $x\in\xO_{\frac{R}{2}},$ $\nabla v(y)=R\,\nabla u(x),$ $D_{i,j}v=R^2D_{i,j}u,$ $R=|x|$ and the above estimate, the result follows at once.
\end{proof}

In the rest of this paper we assume  that the Lipschitz spherical cap $\xo(\tau)$ has the property:

\textit{there exists $\widetilde{\xe}>0,$ such that for any $p\in(\sup\limits_{\tau\in\mathbb{R}}p^*(\tau),\sup\limits_{\tau\in\mathbb{R}}p^*(\tau)+\widetilde{\xe}),$ there exists a solution $u_1$ of theorem \ref{mainprop}. Thus $\xe(\tau)$ is a smooth bounded function with bounded derivatives and there exist $\xe_0, \xe_1>0$ such that $\xe_0\leq\xe(\tau)\leq\xe_1,\;\forall \tau\in\mathbb{R}.$}

\

Now, we recall some facts from the proof of theorem \ref{mainprop}.
Let $a(\tau,t)$ be  the solution of the problem
\be
\partial^2_ta+ A\partial_ta-\xe(\tau) a+\xm(\tau) a^p=0,\label{ode}
\ee
where $A=-\left(n-2\frac{p+1}{p-1}\right),$ $\xe(\tau)=\xl(\tau)+\frac{2}{p-1}(n-\frac{2p}{p-1}),$ $\xm(\tau)=\int_{\xo(\tau)}\xf_1^{p+1}(\tau,\theta)\rd\theta$ and $\xm(\tau) a^{p-1}_\infty(\tau)=\xe(\tau).$ Recall also that we have chosen $a(\tau,t)$ such that
$$a(\tau,0)=\frac{1}{2}a_\infty(\tau),\qquad
\lim_{t\rightarrow\infty}a(\tau,t)=a_\infty(\tau),\quad\hbox{and }\quad \lim_{t\rightarrow-\infty}a(\tau,t)=0.$$

We next prove the following lemma
\begin{lemma}
Let $a$ be the solution of (\ref{ode}), $\xe_0=\inf\limits_{\tau\in\mathbb{R}}\xe(\tau),$
$$\widetilde{\xd}^+(\tau)=\frac{-A+\sqrt{A^2-4(p-1)\xe(\tau)}}{2}\qquad\mathrm{and}\qquad \xd^-(\tau)=\frac{-A+\sqrt{A^2+4\xe(\tau)}}{2}.$$

Then there exists $\widetilde{t}>0$ such that
\begin{eqnarray*}
\left|\frac{\partial a}{\partial \tau}(\tau,t)\right|&\leq& C(\xe_0,p,n)|t|e^{\xd^-(\tau)t},\qquad\forall (\tau,t)\in\mathbb{R}\times(-\infty,-\frac{\widetilde{t}}{\xe_0}),\\[3mm]
\left|\frac{\partial^2a}{\partial \tau^2}(\tau,t)\right|&\leq& C(\xe_0,p,n)|t|^2e^{\xd^-(\tau)t},\qquad \forall (\tau,t)\in\mathbb{R}\times(-\infty,-\frac{\widetilde{t}}{\xe_0})\\[3mm]
\left|\frac{\partial a}{\partial \tau}(\tau,t)\right|&\leq &C(\xe_0,p,n)|t|e^{\widetilde{\xd}^+(\tau)t},\qquad \forall (\tau,t)\in\mathbb{R}\times(\frac{\widetilde{t}}{\xe_0},\infty),\\[3mm]
\left|\frac{\partial^2a}{\partial \tau^2}(\tau,t)\right|&\leq& C(\xe_0,p,n)|t|^2e^{\widetilde{\xd}^+(\tau)t},\qquad\forall (\tau,t)\in\mathbb{R}\times(\frac{\widetilde{t}}{\xe_0},\infty).
\end{eqnarray*}

And

\begin{align}
\nonumber
\left|\frac{\partial a}{\partial \tau}(\tau,t)\right|&\leq C(\xe_0,p,n),\qquad \forall (\tau,t)\in\mathbb{R}\times[-\frac{\widetilde{t}}{\xe_0},\frac{\widetilde{t}}{\xe_0}],\\ \nonumber
\left|\frac{\partial^2a}{\partial \tau^2}(\tau,t)\right|&\leq C(\xe_0,p,n),\qquad\forall (\tau,t)\in\mathbb{R}\times[-\frac{\widetilde{t}}{\xe_0},\frac{\widetilde{t}}{\xe_0}].
\end{align}\label{odereg}

\end{lemma}
\begin{proof}
By our assumptions and lemma 2.5 in \cite{pino} there exists a constant $\overline{t}<0$ (independent on $p,$  $\xm$ and $\tau$) such that
$$\frac{1}{2}e^{\xd^-(\tau)t}\leq \frac{a(\tau,t)}{a_\infty(\tau)}\leq e^{\xd^-(\tau)t},\qquad\forall t\leq\frac{\overline{t}}{\xe_0},$$
where
$$\xd^-(\tau)=\frac{-A+\sqrt{A^2+4\xe(\tau)}}{2}.$$
Choose  $\tau_0\in\mathbb{R}$ and set $a(\tau,t)=a_\infty(\tau)(e^{\xd^-(\tau)t}+w(\tau,t)).$ Then $w$ is a solution of the fixed point problem
\begin{align}\nonumber
w&=-\xe e^{\xd^-(\tau)t}\int_{-\infty}^te^{-2\xd^-(\tau)\xz-A\xz}\left(\int_{-\infty}^\xz e^{\xd^-(\tau)s+As}\left(e^{\xd^-(\tau)s}+w\right)^p\rd s\right)\rd\xz\\
&:=T[w].\label{fix}
\end{align}
Indeed, let $1<p_0<p$ and $\rho$ be sufficiently  small  such that for any $\tau\in O_{\tau_0}=\{\tau\in\mathbb{R}:\;|\tau-\tau_0|<\rho\}$ we have
$$p\xd^-(\tau)\geq p_0\xd^-(\tau_0)\qquad\mathrm{and}\qquad p\xd^-(\tau_0)\geq p_0\xd^-(\tau).$$
Thus, it is easy to find a fixed point in the set of functions defined in $(-\infty,\frac{\overline{t}}{\xe_0})$ and
satisfying $$|w|\leq\frac{1}{2}e^{p_0\xd^-(\tau_0)t}$$ provided $|\overline{t}|$ is fixed large enough (independent of p and $\tau$).

Now let
$$G=\{g:(-\infty,\frac{\overline{t}}{\xe_0})\mapsto\mathbb{R}:\;||e^{-p_0\xd^-(\tau_0)t}g||_{L^\infty(-\infty,\frac{\overline{t}}{\xe_0})}<C\}$$ and define  $F(\tau,g)=g-T(g).$ By (\ref{fix}) we can apply the Implicit Function theorem in the domain $O_{\tau_0}\times G$ to obtain that there exists a unique function $w$ such that
$$F(\tau,w(\tau,t))=0\qquad \hbox{ for any } \quad |\tau-\tau_0|<\rho_0<\rho$$
for some  $\rho_0$ small enough. On the other hand  since $T(g)$ is smooth with respect $\tau$ we have that $w(\tau,t)$ is smooth with respect $\tau.$

\

Notice that $$0=F_\tau(\tau,w(\tau,t))+F_g(\tau,w(\tau,t))\frac{\partial w}{\partial \tau}$$
 thus we have
\be
\left|\frac{\partial w}{\partial \tau}(\tau,t)\right|\leq C(\xe_0,p,n)|t|e^{\xd^-t},\label{bound1}
\ee

provided $|\overline{t}|$ is fixed large enough. Similarly we have
\be
\left|\frac{\partial ^2w}{\partial \tau^2}(\tau,t)\right|\leq C(\xe_0,p,n)|t|^2e^{\xd^-t}.\label{bound2}
\ee
By (\ref{fix}) and the above inequalities we have that the derivatives $\frac{\partial^2 w}{\partial \tau\partial t},\;\frac{\partial^3 w}{\partial^2 \tau\partial t}$ exist and are bounded.

Since the choice of $\tau_0$ is abstract, we conclude that the functions $a,\;\partial_ta\in C^2$ with respect $\tau,$ for any $t\leq\frac{\overline{t}}{\xe_0}.$ We  also have
\begin{align}
\nonumber
\left|\frac{\partial a}{\partial \tau}(\tau,t)\right|&\leq C(\xe_0,p,n)|t|e^{\xd^-(\tau)t},\qquad\forall (\tau,t)\in\mathbb{R}\times(-\infty,-\frac{\widetilde{t}}{\xe_0}),\\
\left|\frac{\partial^2a}{\partial \tau^2}(\tau,t)\right|&\leq C(\xe_0,p,n)|t|^2e^{\xd^-(\tau)t},\qquad \forall (\tau,t)\in\mathbb{R}\times(-\infty,-\frac{\widetilde{t}}{\xe_0}).\label{in1}
\end{align}

Let $t_0\in\left(-\infty,\frac{\overline{t}}{\xe_0}\right)$ such that $a(\tau,t_0),\;\frac{\partial a(\tau,t_0)}{\partial t}\in C^2$ with respect $\tau.$ Using standard ODE techniques we can prove that, if $|h|$ is sufficiently small then
\be
|a(\tau,t)-a(\tau+h,t)|\leq C(t)h,\qquad\forall t\in\mathbb{R},\label{in2}
\ee
where $C(t)$ is a positive smooth function such that $\lim\limits_{t\rightarrow\infty}C(t)=\infty.$

Choose $|h|$ sufficiently small and set $v_h=\frac{a(\tau+h,t)-a(\tau,t)}{h}$ and $a(\tau)=a(\tau,t).$ Then $v_h$ satisfies
\begin{eqnarray}\label{main}
&&\frac{\partial^2v_h}{\partial t^2}+A\frac{\partial v_h}{\partial t}-\xe(\tau+h) v_h=-\xm(\tau+h)\frac{a^p(\tau+h)-a^p(\tau)}{h}\nonumber \\ \nonumber
&&\qquad\qquad-\frac{\xm(\tau+h)-\xm(\tau)}{h}a^p(\tau)
+\frac{\xe(\tau+h)-\xe(\tau)}{h}a(\tau),\qquad\mathrm{in}\;(t_0,\infty),\\
&&v_h(\tau,t_0)= \frac{a(\tau+h,t_0)-a(\tau,t_0)}{h},\\
&&\frac{\partial v_h(\tau,t_0)}{\partial t}=\frac{\frac{\partial a(\tau+h,t_0)}{\partial t}-\frac{\partial a(\tau,t_0)}{\partial t}}{h}. \nonumber
\end{eqnarray}
Using the following expansion
\begin{align}\nonumber
a^p(\tau+h)&=a^p(\tau)+pa^{p-1}(\tau,t)\left(a(\tau+h)-a(\tau)\right)\\ \nonumber
&+\frac{1}{2}\int_{a(\tau)}^{a(\tau+h)}p(p-1)t^{p-2}(a(\tau+h)-t)\rd t,
\end{align}
thus by  the properties of initial data in (\ref{main}), our assumptions on $\xm,\;\xe,$  (\ref{in2}) and above equality, we can obtain by using standard ODE techniques in (\ref{main}) that
$$|v_h|,\quad|\frac{\partial v_h}{\partial t}|<C(t),$$
 where $C(t)$ is a positive smooth function such that $\lim\limits_{t\rightarrow\infty}C(t)=\infty.$ Thus by Arzela Ascoli theorem, there exist a subsequence $\{v_{h_n}\}$ such that $v_{h_n}\rightarrow v$ locally uniformly and $v$ satisfies
\begin{align}
\nonumber
&\frac{\partial^2v}{\partial t^2}+A\frac{\partial v}{\partial t}-\xe(\tau) v=-\xm(\tau)pa^{p-1}(\tau,t)v-\xm'(\tau)a^p(\tau)+\xe'(\tau)a(\tau)\quad\hbox{in} \;(t_0,\infty)\\[3mm] \nonumber
&v(\tau,t_0)= \frac{\partial a(\tau,t_0)}{\partial \tau}\\[3mm] \nonumber
&\frac{\partial v(\tau,t_0)}{\partial t}=\frac{\partial^2 a(\tau,t_0)}{\partial \tau\partial t}.
\end{align}
By uniqueness of the above problem, we have that $\lim\limits_{h\rightarrow0}v_h=v$ for all $\tau\in\mathbb{R}$ and $t\geq t_0.$ And thus $\frac{\partial}{\partial \tau}a(\tau,t)$  exists for any $(\tau,t)\in \mathbb{R}^2.$  Applying  the same argument we can obtain also that $\frac{\partial^2}{\partial \tau^2}a(\tau,t)$ exists for any $(\tau,t)\in \mathbb{R}^2.$ The only difference is that we should use the fact that $ a(\tau,t)>c>0$ for any $(\tau,t)\in\mathbb{R}\times(t_0,\infty).$

\

Set $a=a_\infty w$ then $w$ satisfies
\be
\partial^2_tw+ A\partial_tw-\xe(\tau) w+\xe(\tau) w^p=0.\label{ode*}
\ee
Let us now recall some facts from lemma 2.5 in \cite{pino}.
Set
$$\widetilde{\xd}^+(\tau)=\frac{-A+\sqrt{A^2-4(p-1)\xe(\tau)}}{2}\;\;\mathrm{and}\;\;\widetilde{\xd}^-(\tau)=\frac{-A-\sqrt{A^2-4(p-1)\xe(\tau)}}{2}.$$
There exists a $\widehat{t}>0$ (independent  on $p$ and $\tau$) such that , $\forall \,t\,\geq\,\frac{\widehat{t}}{\xe_0}$
\begin{align}\nonumber
&\frac{1}{2}e^{\widetilde{\xd}^-(\tau)t}\leq 1-w(\tau,t)\leq 2e^{\widetilde{\xd}^-(\tau)t}\\
&\frac{1}{C(\xe_0)}w(1-w)\leq\frac{\partial w}{\partial t}\leq C(\xe_0) w(1-w).\label{2}
\end{align}

Notioce that the function $\frac{\partial w}{\partial \tau}$ is a solution of
\be
\frac{\partial^2v}{\partial t^2}+A\frac{\partial v}{\partial t}-\xe(\tau) v+pw^{p-1}(\tau,t)v=\xe'(\tau)w^p(\tau)+\xe'(\tau)w(\tau),\label{1}
\ee
but the function $\frac{\partial a}{\partial t}$ is one solution of the corresponding  homogeneous problem. For the other solution of the homogeneous problem $\psi$ we can easily prove by using (\ref{2}) that
$$|\psi(t,\tau)|\leq C(\xe_0) e^{\widetilde{\xd}^-(\tau)t}.$$

Thus by the representation formula and the properties of $w,$ we can easily get
$$\left|\frac{\partial w}{\partial \tau}\right|\leq C(\xe_0,p,n)|t|e^{\widetilde{\xd}^+(\tau)t},\qquad  \forall\;t\geq\frac{\widetilde{t}}{\xe_0}.$$
Using the estimates (\ref{2}) and the fact that $w$ is a solution of (\ref{ode*}), we can prove that
$$\left|\frac{\partial^2w}{\partial t^2}\right|< C(\xe_0,n,p)e^{\widetilde{\xd}^+(\tau)t}.$$

Setting $w=\left(1-e^{\widetilde{\xd}^+(\tau)t}+v\right)$, then $v$ can be written (see appendix in \cite{pino})
\begin{align}\nonumber
v&=\xe e^{\widetilde{\xd}^-(\tau)t}\int_{t_p}^te^{-2\widetilde{\xd}^-(\tau)\xz-A\xz}\left(\int_{\xz}^\infty e^{\widetilde{\xd}^-(\tau)s+As}\mathcal{Q}\left(-e^{\widetilde{\xd}^+(\tau)s}+v\right)\rd s\right)\rd\xz\\
&+\xl_pe^{\widetilde{\xd}^-(\tau)t}.\label{fix1}
\end{align}
where $Q(x)=|1+x|^p-1-px,$ $t_p$ is large enough and $\xl_p(\tau)$ is a smooth bounded function. Thus by (\ref{fix1}) and the definition of $v$ we can prove that
there exists a constant $C>0$ such that
$$\frac{1}{C}e^{\widetilde{\xd}^+(\tau)t}\leq-\partial_t^2w(\tau,t)\leq C e^{\widetilde{\xd}^+(\tau)t},\qquad\forall t\geq t_p.$$
By the same argument  we can prove that
$$\left|\frac{\partial^2w}{\partial \tau^2}(\tau,t)\right|\leq C(\xe_0,p,n)|t|^2e^{\widetilde{\xd}^+(\tau)t},\qquad  \forall\;t\geq\frac{\widetilde{t}}{\xe_0}.$$
This ended the proof.
\end{proof}

\begin{lemma}
Let $u_1$ be the solution given by theorem \ref{mainprop}, then the following estimates hold
$$|\partial_\tau u_1(\tau,x)|\leq C|x|^{-\frac{2}{p-1}}\qquad\mathrm{and}\qquad|\partial_\tau^2 u_1(\tau,x)|\leq C|x|^{-\frac{2}{p-1}},$$
where the constant $C$ does not depend on $\tau$ and $x.$\label{mainlem}
\end{lemma}
\begin{proof}
In view of the proof of theorem \ref{mainprop},
$$u_1=|x|^{-\frac{2}{p-1}}f(\tau,\theta)=|x|^{-\frac{2}{p-1}}\left(a(\tau,t)\xf_1(\tau,\theta)+\psi(\tau,\theta)\right),$$
where $\psi$ is a solution of the fixed point problem
\be
\psi=-G_p(\mathcal{M}(\xf_0)+\mathcal{Q}(\psi)),\label{meq}
\ee
where $\xf_0(\tau,\theta)=a(\tau,t)\xf_1(\tau,\theta),$ $\mathcal{M}(\xf_0)=a^p\left(\xf_1^p-\xm\xf_1\right)$ and $$\mathcal{Q}(\psi)=|\xf_0+\psi|^p-\xf^p_0-p\xf^{p-1}_0\psi.$$

We recall here that $|\psi(t,\theta)|<<a(\tau,t)\xf_1(\tau,\theta).$

Here we will only treat  the case $n\geq3$.  For $n=2$ the proof is the same.

By uniqueness, our assumptions on $\xo(\tau),$ and remark \ref{remark}. $\psi=\psi(t,\widetilde{s}),$ $\widetilde{s}\in (0,\xb(\tau)),\;\theta_1=\cos\widetilde{s},$ where $\xb(\tau)$ is a positive smooth function such that $$0<\inf_{\tau\in\mathbb{R}}\xb(\tau)\leq\sup_{\tau\in\mathbb{R}}\xb(\tau)<\pi.$$ Then $\psi$ satisfies
\begin{align}\nonumber
&\left(\partial^2_t
+A\partial_t-\xe(\tau)\right)\psi+\sin^{2-n} (\widetilde{s})\partial_{\widetilde{s}}\left(\sin^{n-2} (\widetilde{s})\partial_{\widetilde{s}}\psi\right)+\xl(\tau)\psi +p\xf_0^{p-1}\psi\\ \nonumber
&=-\mathcal{M}(\xf_0)-{Q}(\psi),
\end{align}
for any $(t,\widetilde{s})\in\mathbb{R}\times(0,\xb(\tau)),$ and $\psi(t,\xb(\tau))=0.$

Setting now $s=\frac{\widetilde{s}}{\xb(\tau)},$ we have that $\psi(\tau,t,s)$ satisfies
\begin{align}\nonumber
\widetilde{L}_p\psi:&=\left(\partial^2_t+A\partial_t-\xe(\tau)\right)\psi+\frac{1}{\xb^2(\tau)}\partial_s^2\psi\\ &+(n-2)\frac{\cos(\xb(\tau)s)}{\xb(\tau)\sin(\xb(\tau)s)}\partial_s\psi+\xl\psi
+p\xf_0^{p-1}\psi=-\mathcal{M}(\xf_0)-{Q}(\psi),\label{eq}
\end{align}
for any $(t,s)\in\mathbb{R}\times(0,1),$ and $\psi(\tau,t,1)=0.$

Let $1<p_0<p$ such that $p-p_0$ is small enough and  let $g:\mathbb{R}\times(0,1)\rightarrow\mathbb{R}$ such that $g\in C^a(\mathbb{R}\times[0,1])$
for some $0<a\leq1,$ and
$$\sup_{\tau\in\mathbb{R}}\;\sup_{(t,s)\in\mathbb{R}\times(0,1)}|a^{-p}(\tau,t)g(t,s)|<\infty.$$

Let $u(\tau,t,s)=-\widetilde{G}_p(\mathcal{M}(\xf_0)+\mathcal{Q}(g))$
be the solution of (\ref{eq}). This solution exists since problem (\ref{eq}) is equivalent to (\ref{meq}). In addition, by proposition \ref{propest} we have the following estimate
\begin{align}\nonumber
\sup_{(t,s)\in\mathbb{R}\times(0,1)}\left|d^{-1}a^{-p_0}(\tau,t)u(\tau,t)\right|
&\leq C\sup_{(t,s)\in\mathbb{R}\times(0,1)}\left|a^{-p_0}(\tau,t)\mathcal{M}(\xf_0)(\tau,t,s)\right|\\ &
+\frac{C}{\xe}\sup_{(t,s)\in\mathbb{R}\times(0,1)}\left|a^{-p_0}(\tau,t)\mathcal{Q}(g)(\tau,t)\right|,\label{est}
\end{align}
for some constant $C>0$ which does not depend on $\tau.$

We can easily prove that
$$\lim_{h\rightarrow0}\sup_{(t,s)\in\mathbb{R}\times(0,1)}|u(\tau+h,t,s)-u(\tau,t,s)|=0.$$

Recall the definitions $$u_h(\tau,t,s)=\frac{u(\tau+h,t,s)-u(\tau,t,s)}{h},\qquad  u(\tau)=u(\tau,t,s), \cdots$$
Clearly $u_h$ satisfies
\begin{align}\nonumber
&\left(\partial^2_t+A\partial_t-\xe(\tau+h)\right)u_h(\tau)+\frac{1}{\xb^2(\tau+h)}\partial_s^2u_h(\tau)\\ \nonumber &+\frac{(n-2)\cos(\xb(\tau+h)s)}{\xb(\tau+h)\sin(\xb(\tau+h)s)}\partial_s u_h(\tau)
+\xl(\tau+h) u_h+p\xf_0^{p-1}(\tau+h)u_h(\tau)\\[3mm] \nonumber
&=-\frac{\frac{1}{\xb^2(\tau+h)}-\frac{1}{\xb^2(\tau)}}{h}\partial_s^2u(\tau)+\frac{\xe(\tau+h)-\xe(\tau)}{h}u(\tau)-\frac{\xl(\tau+h)-\xl(\tau)}{h}u(\tau)\\[3mm] \nonumber
&-(n-2)\frac{\frac{\cos(\xb(\tau+h)s)}{\xb(\tau+h)\sin(\xb(\tau+h)s)}-\frac{\cos(\xb(\tau)s)}{\xb(\tau)\sin(\xb(\tau)s)}}{h}\partial_s u(\tau)-p\frac{\xf_0^{p-1}(\tau+h)-\xf_0^{p-1}(\tau)}{h}u(\tau)\\[3mm] \nonumber
&-\frac{\mathcal{M}(\xf_0)(\tau+h)-\mathcal{M}(\xf_0)(\tau)}{h}-\frac{\mathcal{Q}(g)(\tau+h)-\mathcal{Q}(g)(\tau)}{h}.
\end{align}

Now notice that $u(\tau,t,s)=w(t,\cos(s\xb(\tau)))=v(\tau,x),$ where $x_1=|x|\cos(s\xb(\tau)).$ In addition, $v(\tau,x)$ satisfies
\begin{align}
\begin{cases}
&-\xD_x v+\frac{4}{p-1}\frac{x\cdot \nabla_x v}{|x|^2}+\frac{2}{p-1}\left(n-\frac{2}{p-1}-2\right)\frac{v}{|x|^2}-p\frac{\xf_0^{p-1}v}{|x|^2}=-\frac{g}{|x|^2},\quad\hbox{in}\; C_\xo(\tau)\\ \nonumber
&v=0\qquad\mathrm{in}\;\; \partial C_\xo(\tau)\setminus\{0\}.
\end{cases}
\end{align}
Thus by lemma \ref{lemma11} we have
$$\left|\frac{1}{\sin s\xb(\tau)}\frac{\partial u}{\partial s}\right|\leq \frac{1}{\inf\limits_{\tau\in\mathbb{R}}\xb(\tau)}|x||v_{x_1}|<C.$$
Similarly we can obtain  $\left|\frac{\partial^2u}{\partial s^2}\right|<C$ for some constant $C>0$ which  does not depend on $\tau.$

\

Thus we have
\begin{align}\nonumber
\sup_{(t,s)\in\mathbb{R}\times(0,1)}\left|\frac{1}{\sin s\xb(\tau)}\frac{\partial u}{\partial s}(\tau,t,s)\right|&<C\\
\sup_{(t,s)\in\mathbb{R}\times(0,1)}\left|\frac{\partial^2u}{\partial s}(\tau,t,s)\right|<C,\label{mest}
\end{align}
where the constant $C>0$ does not depend on $\tau.$
Now we have
\begin{align}\nonumber
&\lim_{h\rightarrow0}\sup_{\tau\in\mathbb{R}}
\left|\frac{\frac{\cos(\xb(\tau+h)s)}{\xb(\tau+h)\sin(\xb(\tau+h)s)}-\frac{\cos(\xb(\tau)s)}{\xb(\tau)\sin(\xb(\tau)s)}}{h}\partial_su(\tau)\right|\\ \nonumber
&=\sup_{\tau\in\mathbb{R}}\left|\left(-\frac{\xb'(\tau)}{\xb^2(\tau)}\cot(\xb(\tau)s)-\frac{s\xb'(\tau)}{\sin^2\xb(\tau)s
}\right)\partial_s u(\tau)\right|<C,
\end{align}
where in the last inequality we have used the fact that $$0<\inf_{\tau\in\mathbb{R}}\xb(\tau)\leq\sup_{\tau\in\mathbb{R}}\xb(\tau)<\pi$$ and (\ref{mest}).
Using the fact that
\begin{align}\nonumber
&a^p(\tau+h)\xf_1^p(\tau+h)-a^p(\tau)\xf_1^p(\tau)\\ \nonumber
&=\left(a^p(\tau+h)-a^p(\tau)\right)\xf_1^p(\tau+h)+a^p(\tau)\left(\xf_1^p(\tau+h)-\xf_1^p(\tau)\right),
\end{align}
and
\begin{align} \nonumber
a^p(\tau+h)&=a^p(\tau)+pa^{p-1}(\tau)(a^p(\tau+h)-a^p(\tau))\\ \nonumber
&+\frac{p(p-1)}{2}\int_{a^p(\tau)}^{a^p(\tau+h)}t^{p-2}(a^p(\tau+h)-t)\rd t,
\end{align}
(the same for $\xf_1$), and lemmas \ref{f1}, \ref{odereg}, we have that
$$\left|\lim_{h\rightarrow0}\frac{\mathcal{M}(\xf_0)(\tau+h)-\mathcal{M}(\xf_0)(\tau)}{h}\right|=\left|\frac{\partial\mathcal{M}(\xf_0)}{\partial \tau}\right|<C.$$
Similarly we have that
$$\left|\lim_{h\rightarrow0}\frac{\mathcal{Q}(g)(\tau+h)-\mathcal{Q}(g)(\tau)}{h}\right|=\left|\frac{\partial\mathcal{Q}(g)}{\partial \tau}\right|<C.$$

By proposition \ref{propest} we have
 $$\sup_{\tau\in\mathbb{R}}\sup_{(t,s)\in\mathbb{R}\times(0,1)}|u_h|<C$$
and thus by Arzela Ascoli theorem, there exist a subsequence $\{u_{h_n}\}$ such that $u_{h_n}\rightarrow v$ locally uniformly and $v(\tau,t,s)$ satisfies
\begin{align}\nonumber
&\left(\partial^2_t+A\partial_t-\xe(\tau)\right)v+\frac{1}{\xb^2(\tau)}\partial_s^2v +\frac{\cos(\xb(\tau)s)}{\xb(\tau)\sin(\xb(\tau)s)}\partial_s v\\ \nonumber
&+\xl(\tau) u+p\xf_0^{p-1}(\tau)v=H(\xf_1,a,g),
\end{align}
with $v(\tau,t,1)=0.$ Notice that $$\sup_{\tau\in\mathbb{R}}\sup_{(t,s)\in\mathbb{R}\times(0,1)}|H(\tau,t,s)|<C,$$ thus by proposition \ref{propest} $v$ is a unique solution. Furthermore, $$\lim_{h\rightarrow0}u_h=v=\frac{\partial u}{\partial \tau},$$
and
\be
\sup_{\tau\in\mathbb{R}}\sup_{(t,s)\in\mathbb{R}\times(0,1)}\left|\frac{\partial u}{\partial \tau}(\tau,s,t)\right|<C,\label{par1}
\ee
for some constant  $C$  independent on $g.$

Similarly as (\ref{mest}) we can prove,
\begin{align}\nonumber
&\sup_{\tau\in\mathbb{R}}\sup_{(t,s)\in\mathbb{R}\times(0,1)}\left|\frac{1}{\sin s\xb(\tau)}\frac{\partial^2u}{\partial \tau\partial s}(\tau,t,s)\right|<C\\ \nonumber
&\sup_{\tau\in\mathbb{R}}\sup_{(t,s)\in\mathbb{R}\times(0,1)}\left|\frac{\partial^3u}{\partial \tau\partial s\partial s}(\tau,t,s)\right|<C
\end{align}
and by the same argument as above
\be
\sup_{\tau\in\mathbb{R}}\sup_{(t,s)\in\mathbb{R}\times(0,1)}\left|\frac{\partial^2 u}{\partial \tau\partial \tau}(\tau,t,s)\right|<C,\label{par2}
\ee
where $C$ is a constant which depends on $g.$

Now we consider the fix point problem (\ref{eq}). Let $\tau_0\in\mathbb{R}$ and $\rho$ be small enough such that for any $\tau\in O_{\tau_0}=\{\tau\in\mathbb{R}:\;|\tau-\tau_0|<\rho\}$ we have
$p\xd^-(\tau)\geq p_0\xd^-(\tau_0),$ where
$$\xd^-(\tau)=\frac{-A+\sqrt{A^2+4\xe(\tau)}}{2}.$$
We can easily show that $a^p(\tau,t)\leq Ca^{p_0}(\tau_0,t),\;\forall \tau\in O_{\tau_0},$ for some positive constant $C$ independent on  $\tau$ and $t.$

Now since $0<p-p_0$ is small enough, we can use a fix point argument like in \cite{pino} (see remark \ref{remark}) in the Banach space
$$\mathbf{X}=\{g\in L^\infty\left(\mathbb{R}\times(0,1)\right):\;\sup_{(t,s)\in\mathbb{R}\times(0,1)} |a^{-p_0}(\tau_0,t)g(t,s)|<\infty \}$$
to prove that there exists a unique solution $$\psi(\tau,t,s)=-\widetilde{G}_p(\mathcal{M}(\xf_0)+\mathcal{Q}(\psi(\tau,t,s))),\qquad\forall \tau\in O_{\tau_0}.$$
Now, let $(\tau,g)\in O_{\tau_0}\times \mathbf{X},$ we set the bounded operator
$$T(\tau,g)=g+\widetilde{G}_p(g),$$
We can apply the Implicit Function theorem to $O_{\tau_0}\times \mathbf{X}$ to obtain that:

let $0<\xr_0\leq\xr$ be small enough, then for any $\tau\in\{\tau\in\mathbb{R}:\;
|\tau-\tau_0|<\xr_0\}\subset O_{\tau_0}$ there exists a function $\psi(\tau,t,s)$ such that
$$T(\tau,\psi(\tau,t,s))=0.$$
Using  (\ref{par1}),  (\ref{par2}) and again the Implicit Function theorem, we can also prove that $\partial_\tau\psi,\;\partial_\tau^2\psi$ exist.  Furthermore using the fact that
$$0=T_\tau(\tau,\psi(\tau))+T_g(\tau,\psi(\tau))\partial_\tau\psi,$$
and the estimate (\ref{par1}) we have that
$$\sup_{\tau\in(\tau_0-\xr_0,t_0+\xr_0)}\sup_{(t,s)\in\mathbb{R}\times(0,1)}\left|\frac{\partial u}{\partial \tau}(\tau,t,s)\right|<C.$$
Similarly we have
$$\sup_{\tau\in(\tau_0-\xr_0,t_0+\xr_0)}\sup_{(t,s)\in\mathbb{R}\times(0,1)}\left|\frac{\partial^2 u}{\partial \tau\partial \tau}(\tau,t,s)\right|<C.$$
And the result follows since $\tau_0$ is abstract.
\end{proof}

\setcounter{equation}{0}
\section{The proof of theorems \ref{th1.2} and \ref{th1.3}}\label{theorems}
Let $x\in\mathbb{R}^n,\;n\geq2,$ $R>0,$ $B_R(0)\subset\mathbb{R}^n$ and $$r_{\xs(\tau)}=|x-\xs(\tau)|,$$
where $\xs:\mathbb{R}\rightarrow\mathbb{R}^n$ is a smooth curve such that
$$\sup_{\tau\in\mathbb{R}}\left\{|\xs(\tau)|+|\xs'(\tau)|+|\xs''(\tau)|\right\}<C<\infty.$$
Define
$$\widetilde{r}^2=\sum_{i=1}^n|(x_i-\sup|\xs(\tau)|)^2:$$

Given $\tau,$ let $(r_{\xs(\tau)}, \theta) \in[0,\infty)\times\mathbb{S}^{n-1}$ be the spherical-coordinates of $x\in\mathbb{R}^n$ centered at $\xs(\tau)$ abbreviated
by $x = (r_{\xs(\tau)}, \theta).$ We define the cone
$$\widetilde{C}_{\xo(\tau)}=\{x=(r_{\xs(\tau)},\theta):\;r_{\xs(\tau)}>0,\;\theta\in\xo(\tau)\}\subset\mathbb{R}^n.$$
and we denote by
$$\xO_{\tau_1,\tau_2}=\{(\tau,x)\in(\tau_1,\tau_2)\times\mathbb{R}^n: x\in \widetilde{C}_{\xo(\tau)}\}\subset\mathbb{R}^{n+1}.$$

$$\xO_{\tau_1,\tau_2}^R=\xO_{\tau_1,\tau_2}\cap \{(\tau,x)\in(\tau_1,\tau_2)\times\mathbb{R}^{n}:\;x\in B_R(\xs(\tau))\}\subset\mathbb{R}^{n+1},$$ and
$$S_{\tau_1,\tau_2}=\{(\tau,x)\in[\tau_1,\tau_2]\times\mathbb{R}^n: r_{\xs(\tau)}=0\}.$$

Let $C_{\xd,\rho}\left(\xO_{\tau_1,\tau_2}^R\right)$ be the set of continuous function $f\in C\left(\xO_{\tau_1,\tau_2}^R\right)$ with norm
\begin{align}\nonumber
&||f||_{C_{\xd,\rho}\left(\xO_{\tau_1,\tau_2}^R\right)}:=\sup_{(\tau,x)\in\xO_{\tau_1,\tau_2}^R}\left(\chi_{[0,1]}(r_{\xs(\tau)})r^{-\xd}_{\xs(\tau)}|f|
+\chi_{[1,\infty)}(r_{\xs(\tau)})\widetilde{r}^{\;-\rho}|f|\right).
\end{align}
Let $\xd\in(-n-\xg+2,\xg),$ we define $\xf_\xd(\tau,\theta)$ to be the unique positive solution of
\begin{align}
\begin{cases}
\xD_{\mathbb{S}^{n-1}}\xf_\xd+\xl\xf_\xd+\left(\xd(\xd+n-2)-\xl\right)\xf_\xd&=-1,\qquad\mathrm{in}\;\xo(\tau)\\ \nonumber
\xf_\xd&=\phantom{-}0,\qquad\mathrm{on}\;\partial\xo(\tau).
\end{cases}
\end{align}
Notice here that $\xl=\xg^2+\xg(n-2),$ thus $\xd(\xd+n-2)-\xl<0$ if and only if $\xd\in(-n-\xg+2,\xg).$ A direct computation shows that
$$-\xD_x\left(|x|^\xd\xf_\xd\right)
=|x|^{\xd-2}.$$
In view of lemma \ref{f1} we have that $\xf_\xd=\xf_\xd(t)$ where $t\in (0,\xb(\tau))$ and it satisfies
\begin{align*}
\begin{cases}
\sin^{2-n} t\frac{d}{dt}\left(\sin^{n-2} t\frac{d\xf_\xd}{dt}\right)+\xl\xf_\xd+\left(\xd(\xd+n-2)-\xl\right)\xf_\xd&=-1\quad\mathrm{in}\;(0,\xb(\tau))\\
\xf_\xd(\xb(\tau))&=\phantom{-}0.
\end{cases}
\end{align*}

We next set $\xb^*=\sup\limits_{\tau\in\mathbb{R}}\xb(\tau),$ and $\xl^*=\inf\limits_{\tau\in\mathbb{R}}\xl(\tau),$ $\xg^*=\inf\limits_{\tau\in\mathbb{R}}\xg(\tau)$ and we let $\xf_\xd^*$ be the solution of
\begin{align}
\begin{cases}
\sin^{2-n} t\frac{d}{dt}\left(\sin^{n-2} t\frac{d\xf_\xd^*}{dt}\right)+\xl^*\xf_\xd^*+\left(\xd(\xd+n-2)-\xl^*\right)\xf_\xd^*&=-1\quad\mathrm{in}\;(0,\xb^*)\\ \nonumber
\xf_\xd(\xb^*)&=\phantom{-}0
\end{cases}
\end{align}
with $\xg\in(-n-\xg^*+2,\xg^*).$

Thus $\xf^*_\xd$ is the unique solution of the problem
\begin{align}
\begin{cases}
\xD_{\mathbb{S}^{n-1}}\xf_\xd^*+\xl^*\xf^*_\xd+\left(\xd(\xd+n-2)-\xl^*\right)\xf_\xd^*&=-1,\qquad\mathrm{in}\;\xo^*\\ \nonumber
\xf_\xd&=\phantom{-}0,\qquad\mathrm{on}\;\partial\xo^*
\end{cases}
\end{align}
where $\xo^*=\bigcup_{\tau}\xo(\tau)$ and by assumptions we have that $\xo^*\subsetneq \mathbb{S}^{n-1}.$
\begin{prop}
Assume that $\xd,\rho\in(-n-\xg^*+2,0],$
and
\be
\sup_{\tau\in\mathbb{R}}\left\{|\xs(\tau)|+|\xs'(\tau)|+|\xs''(\tau)|\right\}<\xe,\label{asu*}
\ee
where $\xe>0$ is small enough. Then, for all $\tau_1<\tau_2\in\mathbb{R},$ and $R>0,$ there exists a unique operator
$$G_{\xd,\rho,R,\tau_1,\tau_2}:\;C_{\xd,\rho}\left(\xO_{\tau_1,\tau_2}^R\right)\rightarrow C_{\xd,\rho}\left(\xO_{\tau_1,\tau_2}^R\right),$$
such that, for each $f\in C_{\xd,\rho}\left(\xO_{\tau_1,\tau_2}^R\right),$ the function  $G_{\xd,\rho,R,\tau_1,\tau_2}(f)$ is a solution of problem
\begin{align}
\begin{cases}
\xD u&=\frac{1}{r_{\xs(\tau)}^2}f,\qquad\mathrm{in}\qquad \xO_{\tau_1,\tau_2}^R,\\
u&=\phantom{\frac{1}{|x|^2}}0,\qquad\mathrm{on}\qquad  \partial\xO_{\tau_1,\tau_2}^R\setminus S_{\tau_1,\tau_2}.\label{3.2}
\end{cases}
\end{align}
Moreover the norm of $G_{\xd,\rho,R,\tau_1,\tau_2}$ is bounded by a constant $c>0$ which does not depend on $R,\;\tau_1$ and $\tau_2.$\label{pro1*}
\end{prop}
\begin{proof}
Without loss of generality we can  assume that $R>4.$

We first solve, for each $r\in(0,\frac{1}{4})$, the problem
\begin{align}
\begin{cases}
\xD u&=\frac{1}{|x-\xs(\tau)|^2}f,\qquad\mathrm{in}\qquad \xO_{\tau_1,\tau_2}^R\setminus \xO_{\tau_1,\tau_2}^r,\\
u&=\phantom{\frac{1}{|x|^2}}0,\qquad\mathrm{on}\qquad  \partial\left(\xO_{\tau_1,\tau_2}^R\setminus \xO_{\tau_1,\tau_2}^r\right).
\end{cases}\label{3.3}
\end{align}
and call $u_r$ its unique solution.

A straightforward calculations show that
$$-\xD (r_{\xs(\tau)}^\xd\xf_\xd^*)\geq r_{\xs(\tau)}^{\xd-2}(1-|\xd|(|\xd|+1)|\xs'|)-|\xd||\xs''|r_{\xs(\tau)}^{\xd-1}.$$
We choose $\xe$ small enough such that
$$
-\xD (r_{\xs(\tau)}^\xd\xf_\xd^*)\geq\frac{1}{2}\left(r_{\xs(\tau)}^{\xd-2}-r_{\xs(\tau)}^{\xd-1}\right).
$$
Let $\psi$ be the solution of
\begin{align*}
\begin{cases}
\xD_{\mathbb{S}^{n-1}}\psi=-C||f||_{C_{\xd,\rho}\left(\xO_{\tau_1,\tau_2}^R\right)}\qquad\mathrm{in }\quad\xo^*&\\
\psi=\phantom{-}0,\qquad\mathrm{on}\qquad\partial\xo*&
\end{cases}
\end{align*}
for some constant  $C>0$ and we define the following cut-of function $\eta:\;\mathbb{R}^n\rightarrow[0,1]$ by  $\eta=1$ in $B_{\frac{1}{2}}(0)\subset\mathbb{R}^n$ and $\eta\in C_0^{\infty}(B_1(0)).$

We next set
$$\xF(\tau,x)=C||f||_{C_{\xd,\rho}\left(\xO_{\tau_1,\tau_2}^R\right)}\eta(x)r_{\xs(\tau)}^\xd\xf_\xd^*+\psi.$$
If we choose the uniform constant $C>0,$ large enough, we have by the maximum principle
\begin{align}
\nonumber
|u_r(\tau,x)|&\leq\xF(\tau,x) \leq C||f||_{C_{\xd,\rho}\left(\xO_{\tau_1,\tau_2}^R\right)}\xf_\xd^*
|x|^{\xd}+\psi\\
&\leq C||f||_{C_{\xd,\rho}\left(\xO_{\tau_1,\tau_2}^R\right)}\xf_\xd^*(\theta)(|x|^{\xd}+1),\qquad\forall (\tau,x)\in \xO_{\tau_1,\tau_2}^R\setminus \xO_{\tau_1,\tau_2}^r \label{un*}
\end{align}
where in the last inequality we have used the fact that
$$\psi(\theta)\leq C||f||_{C_{\xd,\rho}\left(\xO_{\tau_1,\tau_2}^R\right)}\xf_\xd^*(\theta),\qquad\forall \theta\in \xo^*.$$
Using (\ref{un*}) and again the maximum principle we get
\be
|u_r(\tau,x)|\leq C||f||_{C_{\xd,\rho}\left(\xO_{\tau_1,\tau_2}^R\right)}\xf_\xd^*(\theta)|x|^{\xd},\qquad\forall (\tau,x)\in \xO_{\tau_1,\tau_2}^\frac{1}{2}\setminus \xO_{\tau_1,\tau_2}^r.\label{111}
\ee
Set now $\psi_0=\widetilde{r}^\xr\xf_\xr^*,$ then
$$\xD_{\mathbb{S}^{n-1}}\psi_0=-\widetilde{r}^{\xr-2}.$$
Thus using (\ref{111}) and the maximum principle we obtain,
\be
|u_r|\leq C(\sup_{\tau\in\mathbb{R}}|\xs|)||f||_{C_{\xd,\rho}\left(\xO_{\tau_1,\tau_2}^R\right)}||\xf_\xr^*||_{L^\infty(\xo)}
|x|^{\xr},\;\;\;\forall r_{\xs(\tau)}>\frac{1}{2}.\label{un}
\ee

By standard interior elliptic estimates and Arzela Ascoli theorem, there exists a subsequence $\{u_{r_j}\},$ such that $r_j\downarrow0$ and $u_{r_j}\rightarrow u$ locally uniformly. By standard elliptic theory, (\ref{111}) and (\ref{un}), we have that $u\in C^2(\xO_{\tau_1,\tau_2}^R)$ and is unique.
\end{proof}

\begin{proof}[Proof of theorem \ref{th1.2}]
We choose
$\xd=-\frac{2}{p-1}$
 and we set $$u_\xe(x,\tau)=\eta(x)\xe^{-\frac{2}{p-1}}u_1(\frac{x-\xs}{\xe}),$$
where $u_1$ is the function given in theorem \ref{mainprop} and $\eta:\;\mathbb{R}^n\rightarrow[0,1]$ is a cut-of function such that $\eta=1$ in $B_{\frac{1}{2}}(0)\subset\mathbb{R}^n$ and $\eta\in C_0^{\infty}(B_1(0)).$

By construction of $u_1(x)$ and lemma \ref{lemma11} we have
\bea
\nonumber
|\nabla_x u_1(\tau,x)|&\leq& C(n,p,\xl,C_{\xo(\tau)})|x|^{-1}\\
|D^2_xu(\tau,x)|&\leq& C(n,p,\xl,C_{\xo(\tau)})|x|^{-2}.\label{222}
\eea
First we assume that
\be
\sup_{\tau\in\mathbb{R}}\left\{|\xs(\tau)|+|\xs'(\tau)|+|\xs''(\tau)|\right\}<\widetilde{\xe},\label{asu**}
\ee
where $\widetilde{\xe}>0$ is small enough. Then by the above two estimates  (\ref{222}), (\ref{asu**})  and lemma \ref{mainlem}  we have
\be
|\partial^2_{\tau}u_{\xe}(x,\tau)|\leq C r^{-\frac{2}{p-1}}(\tau)+C(n,\xg^*)\widetilde{\xe}\left(r^{-\frac{2}{p-1}-2}_{\xs(\tau)}+r^{-\frac{2}{p-2}-1}_{\xs(\tau)}\right).\label{ass1}
\ee
Now, let $R>4$, $\tau_1<\tau_2\in\mathbb{R}$ and define the following problem
\begin{align}
\begin{cases}
-\xD u&=u^p,\qquad\mathrm{in}\qquad \xO_{\tau_1,\tau_2}^R,\\
u&>0\phantom{^p},\qquad\mathrm{in}\qquad \xO_{\tau_1,\tau_2}^R\\
u&=0\phantom{^p},\qquad\mathrm{on}\qquad  \partial\xO_{\tau_1,\tau_2}^R\setminus S_{\tau_1,\tau_2}.
\end{cases}\label{3.4}
\end{align}
We then look for a solution of the form $u=u_\xe+v$.
By virtue of proposition \ref{pro1*} we can rewrite this equation as the fixed point problem
\be
v=-G_{\xd,\rho,R,\tau_1,\tau_2}\left(|x|^2\left(\xD u_\xe +|u_\xe+v|^p\right)\right)\label{3.4*}
\ee
$$\xD v=-|u_\xe+v|^p-\xD u_\xe.$$
We assume that $\xe$ is small enough, then by (\ref{ass1}) we have for some constant $C_0(n,\xg)>0,$
\begin{align}\nonumber
|||u_\xe|^p+\xD u_\xe||_{C_{\xd,\rho}\left(\xO_{\tau_1,\tau_2}^R\right)}&\leq C_0\left(\xe^{n+\xg-2-\frac{p-3}{p-1}}+\xe^{2}+\xe+\widetilde{\xe}\right)\\ \nonumber
&\leq C_0\left(\xe+\widetilde{\xe}\right),
\end{align}
we recall here that $\xd=-\frac{2}{p-1}.$

Then, using theorem \ref{mainprop} one can easily see that
\begin{align}
\nonumber
&\big\||x|^2|v_\xe+v_1|^p-|v_\xe+v_2|^p\big\|_{C_{\xd,\rho}\left(\xO_{\tau_1,\tau_2}^R\right)}\\[3mm] \nonumber
&\leq C_1(n,\xg^*,p)\left(\sup_{\tau\in\mathbb{R}}||\xf_p||_{L^\infty(\xo)}+\widetilde{\xe}\right)^{p-1}||v_1-v_2||_{C_{\xd,\rho}\left(\xO_{\tau_1,\tau_2}^1\right)}\\[3mm]
&+C(n,\xg^*,p) (\xe+\widetilde{\xe})^{p-1}||v_1-v_2||_{C_{\xd,\rho}\left(\xO_{\tau_1,\tau_2}^R\setminus \xO_{\tau_1,\tau_2}^1\right)},\label{ass3}
\end{align}
for all $v_1,v_2\in C_{\xd,\xb}\left(C_\xo^R\setminus\{0\}\times(\tau_1,\tau_2)\right)$ such that
$$||v_i||_{C_{\xd,\xb}\left(C_\xo^R\setminus\{0\}\times(\tau_1,\tau_2)\right)}\leq 2C_0(\xe+\widetilde{\xe}).$$
We recall that all the constants above do not depend on $R,\; t_1,$ $t_2,$ $\xe$ and $\widetilde{\xe}.$  To obtain a contraction mapping is enough to take $\xe,\;\widetilde{\xe}$ small enough and $p$ close enough to $\sup\limits_{\tau\in\mathbb{R}}p^*$ to ensure that $\sup\limits_{\tau\in\mathbb{R}}||\xf_p(\tau,\cdot)||_{L^\infty(\xo(\tau))}$ is as small as we need. The above estimates allow an application of contraction mapping principle in the ball of radius $2C_0(\xe+\widetilde{\xe})$ in $\xO_{\tau_1,\tau_2}^R$ to obtain a solution to  the problem (\ref{3.4*}), which we denote by $$u_{R,\tau_1,\tau_2}=u_\xe+v_{R,\tau_1,\tau_2}.$$

In view of the fix point argument, we have that $|v_{R,t_1,t_2}|\leq\frac {u_\xe}{4}$ near $S_{\tau_1,\tau_2},$ thus  the solution $u_{R,t_1,t_2}$ is singular along $S_{\tau_1,\tau_2}$ and positive near $S_{\tau_1,\tau_2}.$ The maximum principle then implies that $$u_{R,t_1,t_2}>0\qquad \hbox{in } \quad \xO_{\tau_1,\tau_2}^R.$$
Moreover  we have that
 $$||v_{R,\tau_1,\tau_2}||_{C_{\xd,\xb}\left(\xO_{\tau_1,\tau_2}^R\right)}\leq 2C_0(\xe+\widetilde{\xe}).$$
That is , $v_{R,\tau_1,\tau_2}$ is uniformly bounded by a constant which depend only on $n,\;\xg^*,\;p.$  By standard interior elliptic estimates and Arzela-Ascoli theorem, there exists a subsequence $\{u_{R_j,-\tau_j,\tau_j}\},$ such that $R_j\uparrow\infty,$ $\tau_j\uparrow\infty$ and $u_{R_j,-\tau_j,\tau_j}\rightarrow u$ locally uniformly. Again  standard elliptic theory yields  $u\in C^2(\xO_{-\infty,\infty}).$

For the general case
$$
\sup_{\tau\in\mathbb{R}}\left\{|\xs(\tau)|+|\xs'(\tau)|+|\xs''(\tau)|\right\}<C,
$$
set $\widetilde{\xs}=\frac{\xs}{k},$ where $k>0$ is large enough such that
$$
\sup_{\tau\in\mathbb{R}}\left\{|\widetilde{\xs}(\tau)|+|\widetilde{\xs}'(\tau)|+|\widetilde{\xs}''(\tau)|\right\}<\widetilde{\xe}.
$$
As before we can find a solution $u(x)$ of the problem with singularity along $\{(\tau,x)\in\mathbb{R}\times\mathbb{R}^n:\;|x-\widetilde{\xs}(\tau)|=0\}.$
But the function $v(y)=k^{\frac{2}{p-1}}u(k y),$ where $y=k x,$ is a singular solution of the problem and has singularity along
$S_{-\infty,\infty},$ and the result follows.
\end{proof}

Let $\xa>0,$ $\xO$ be a bounded Lipschitz domain such that
$$\xO\cap\xO_{\tau_1-\xa,\tau_2+\xa}^R=\xO_{\tau_1-\xa,\tau_2+\xa}^R\subset\mathbb{R}^{n+1}.$$

Let $C_{\xd}\left(\xO_{\tau_1,\tau_2}^R\right)$ be the set of continuous function $f\in C\left(\xO_{\tau_1,\tau_2}^R\right)$ with norm
$$
||f||_{C_{\xd}\left(\xO_{\tau_1,\tau_2}^R\right)}=\sup_{(\tau,x)\in\xO_{\tau_1,\tau_2}^R}\left(r^{-\xd}(\tau)|f|\right).
$$

We define $C_\xd(\xO)$ to be the space of the continuous function in $\xO$ with the norm
$$||f||_{C_\xd(\xO)}=||f||_{C_{\xd}\left(\xO_{\tau_1-\xa,\tau_2+\xa}^R\right)}+||f||_{L^\infty\left(\overline{\xO}\setminus\xO_{\tau_1-\frac{\xa}{4},\tau_2+\frac{\xa}{4}}^\frac{R}{2}\right)}. $$

We consider a smooth, positive bounded function
$\xn : \overline{\xO}\rightarrow(0,\infty),$
which is equal to $r_{\xs(\tau)}$ in $\xO_{\tau_1-\frac{\xa}{4},\tau_2+\frac{\xa}{4}}^\frac{R}{2}$ and satisfying
$$0<\sup_{x\in\overline{\xO}\setminus\xO_{\tau_1-\frac{\xa}{2},\tau_2+\frac{\xa}{2}}^R}\xn<C.$$
We obtain the following proposition
\begin{prop}
Let $\tau_1<\tau_2\in\mathbb{R}$ and $\xa>0$ be small enough. Assume that $\xO$ is a bounded Lipschitz domain such that
$$\xO\cap\xO_{\tau_1-2\xa,\tau_2+2\xa}^R=\xO_{\tau_1-2\xa,\tau_2+2\xa}^R\subset\mathbb{R}^{n+1},$$
$\xd\in(-n-\xg^*+2,0]$
and
\be
\sup_{\tau\in\mathbb{R}}\left\{|\xs(\tau)|+|\xs'(\tau)|+|\xs''(\tau)|\right\}<\xe,\label{asu}
\ee
for some  $\xe>0$  small enough. Then,  there exists a unique operator
$$G_{\xd,\tau_1,\tau_2}:\;C_{\xd}\left(\xO\right)\rightarrow C_{\xd}\left(\xO\right),$$
such that, for each $f\in C_{\xd}\left(\xO\right),$ the function  $G_{\xd,\tau_1,\tau_2}(f)$ is a solution of the problem
\begin{align}
\begin{cases}
\xD u&=\frac{1}{\xn^2}f,\qquad\mathrm{in}\qquad \xO,\\
u&=\phantom{\frac{1}{|x|^2}}0,\qquad\mathrm{on}\qquad  \partial\xO\setminus S_{\tau_1-\xa,\tau_2+\xa}.
\end{cases}\label{3.2*}
\end{align}

Moreover the norm of $G_{\xd,\tau_1,\tau_2}$ is bounded by a constant $c>0$ which does not depend on $R,\;\tau_1$ and $\tau_2.$\label{pro1**}
\end{prop}
\begin{proof}
Let $\widehat{\xs}(t)$ be a bounded smooth curve such that
$$
\sup_{\tau\in\mathbb{R}}\left\{|\widehat{\xs}(\tau)|+|\widehat{\xs}'(\tau)|+|\widehat{\xs}''(\tau)|\right\}<2\xe,
$$

$$r_{\widehat{\xs}(\tau)}= r_{\xs(\tau)},\qquad\forall (\tau,x)\in\xO_{\tau_1-\frac{\xa}{4},\tau_2+\frac{\xa}{4}}^R,$$
$$r_{\widehat{\xs}(\tau)}\geq r_{\xs(\tau)},\;\;\forall (\tau,x)\in\xO,$$
and
$$r_{\widehat{\xs}(\tau)}>c>0,\;\;\;\forall (\tau,x)\in \xO_{\tau_1-\xa,\tau_2+\xa}^R\setminus\overline{\xO_{\tau_1-\frac{\xa}{2},\tau_2+\frac{\xa}{2}}^R}.$$
Given $\tau$, we let $\widehat{\xo}(\tau)\subsetneq\mathbb{S}^{n-1}$ be the corresponding Lipschitz spherical cap and  $(r_{\widehat{\xs}(\tau)}, \theta) \in[0,\infty)\times\mathbb{S}^{n-1}$ be the spherical-coordinates of $x\in\mathbb{R}^n$ centered at $\widehat{\xs}(\tau)$ abbreviated
by $x = (r_{\widehat{\xs}(\tau)}, \theta).$

We set
\begin{align}\nonumber
\widehat{C}_{\widehat{\xo}(\tau)}&=\{(r_{\widehat{\xs}(\tau)},\theta):\;\widehat{r}(\tau)>0,\;\theta\in\widehat{\xo}(\tau)\},\\ \nonumber
\widehat{\xO}_{\tau_1,\tau_2}&=\{(\tau,x)\in(\tau_1,\tau_2)\times\mathbb{R}^n: x\in \widehat{C}_{\widehat{\xo}(\tau)}\}
\end{align}
and
$\widehat{\xO}_{\tau_1,\tau_2}^R=\widehat{\xO}_{\tau_1,\tau_2}\cap \{(\tau,x)\in(\tau_1,\tau_2)\times\mathbb{R}^{n}:\;x\in B_R(\widehat{\xs}(\tau))\}\subset\mathbb{R}^{n+1}.$ We construct $\widehat{\xo}(\tau)$ such that $$\xO_{\tau_1-\xa,\tau_2+\xa}^R\subsetneq\widehat{\xO}_{\tau_1-\xa,\tau_2+\xa}^{2R},$$
$$\widehat{\xO}_{\tau_1-\frac{\xa}{4},\tau_2+\frac{\xa}{4}}^R=\xO_{\tau_1-\frac{\xa}{4},\tau_2+\frac{\xa}{4}}^R.$$

We next define  $\eta$ be a cut-off function satisfying  $\eta=1$ in $\xO_{\tau_1-\frac{\xa}{2},\tau_2+\frac{\xa}{2}}^\frac{R}{2}$ and $\eta=0$ in $\xO\setminus\xO_{\tau_1-\xa,\tau_2+\xa}^R.$
We write  $\widehat{f}=\eta f$ and we let $u_1=G_{\xd,\rho,R,\tau_1,\tau_2}(\widehat{f})$ be the function given by proposition \ref{pro1*} in $\widehat{\xO}_{\tau_1-\xa,\tau_2+\xa}^{2R}.$

Set $$\widetilde{f}=f-\xn\xD\left(\eta u_1\right),$$
then $\widetilde{f}$ has support in $\xO\setminus\xO_{\tau_1-\frac{\xa}{4},\tau_2+\frac{\xa}{4}}^\frac{R}{2},$ and $\widetilde{f}\in C(\xO).$ Furthermore we have
$$||\widetilde{f}||_{C_\xd(\xO)}\leq C||f||_{C_\xd(\xO)},$$
for some positive constant $C>0.$

Finally, let $u_2$ be a solution of
\begin{align}\nonumber
\begin{cases}
\xD u&=\frac{1}{\xn^2}\widetilde{f},\qquad\mathrm{in}\qquad \xO,\\
u&=\phantom{\frac{1}{\xn^2}}0,\qquad\mathrm{on}\qquad  \partial\xO,
\end{cases}
\end{align}
which clearly satisfy the bound
$$||u_2||_{L^\infty(\xO)}\leq C||\widetilde{f}||_{C_\xd(\xO)}\leq C ||f||_{C_\xd(\xO)}.$$
The desired result then follows by looking for a  solution of (\ref{3.2*}) of the form  $u = \eta u_1 +u_2.$
\end{proof}

\begin{proof}[Proof of theorem \ref{th1.3}]
We choose
$\xd=-\frac{2}{p-1}$ and we set $$u_\xe(x,\tau)=\eta(x)\xe^{-\frac{2}{p-1}}u_1(\frac{x-\xs}{\xe}),$$
where $u_1$ is the function given by theorem \ref{mainprop} and $\eta:\;\mathbb{R}^n\rightarrow[0,1]$ is a cut-of function  such that $\eta=1$ in $\xO_{\tau_1-\frac{\xa}{2},\tau_2+\frac{\xa}{2}}^\frac{R}{2}$ and $\eta=0$ in $\xO\setminus\xO_{\tau_1-\xa,\tau_2+\xa}^R.$

The rest of the proof is the same as in theorem \ref{th1.2}, the only difference is that we use proposition \ref{pro1**} instead of proposition \ref{pro1*}.
\end{proof}

\ack The author would like to thank Prof. Manuel del Pino for proposing the problem and for useful discussions. Also the author would like to thank Prof. Fethi Mahmoudi for reading this work and for useful comments. Part of this work was done under partial
support from Fondecyt Grant 3140567.

\appendix{Proof of lemma \ref{f1}}
To prove lemma \ref{f1} we need the following inequality whose the proof can be found  in \cite{M} (theorem 2, page 43).
\begin{lemma}
Let $A(r),\;B(r)$ be nonnegative functions such that $1/A(r),\;B(r)$
are integrable in $(r,\infty)$ and $(0,r)$, respectively, for all
positive $r<\infty$. Then, for $q\geq2$ the Sobolev inequality
\be
\bigg[\int_0^s B(t)|u(t)|^q \rd t\bigg]^{1/q}\leq C\bigg[\int_0^s
A(t)|u'(t)|^2\rd t\bigg]^{1/2} \;\;\;\;\;,\label{14}
\ee
is valid for all $u\in C^1[0,s]$ such that $u(s)=0$ (or vanish near
infinity, if $s=\infty$), if and only if
$$
 K=\sup_{r\in(0,s)}\bigg[\int_0^r
B(t)\rd t\bigg]^{1/q}\bigg[\int_r^s
(A(t))^{-1}\rd t\bigg]^{1/2}
$$
is finite. The best constant in
\eqref{14} satisfies the following inequality
$$
K\leq C\leq K\bigg(\frac{q}{q-1}\bigg)^{1/2}q^{1/q}.\label{13**}
$$\label{l1th2}
\end{lemma}

\begin{proof}[Proof of lemma \ref{f1}]
Let $n\geq3$, (for $n=2$ the proof is easy and we omit it). By our assumptions on $\xo(\tau)$ and without loss of generality, we can set $\theta_1=\cos t,$ with $0<t<\xb(\tau),$ where $\xb(\tau)$ is a smooth function with bounded derivatives such that
$$0<\inf\limits_{\tau\in\mathbb{R}}\xb(\tau)<\sup\limits_{\tau\in\mathbb{R}}\xb(\tau)<\pi.$$
 Then problem (\ref{eigen}) is clearly  equivalent to

\begin{align}
\begin{cases}
-\sin^{2-n} t\frac{d}{dt}\left(\sin^{n-2} t\frac{d\xf_1}{dt}\right)&=\xl\xf_1,\qquad\mathrm{in}\;\;(0,\xb(\tau)).\\
\xf_1(\xb(\tau))&=0\\
\partial_t\xf_1(0)&=0.
\end{cases}\label{eigenode}
\end{align}

We denote by $\mathcal{H}((0,\xb(\tau)))$ the completion of $C^\infty([0,\xb(\tau)])$ under the norm
$$||v||_{\mathcal{H}((0,\xb(\tau)))}^2=\int_0^{\xb(\tau)}\sin^{n-2}(t)|\partial_tv|^2\rd t<\infty,$$
and the property $v(\xb(\tau))=\partial_tv(0)=0.$

The space $\mathcal{\mathcal{H}(\xo(\tau))}$ is a Hilbert space with inner product
$$(u,v)=\int_0^{\xb(\tau)}\sin^{n-2}(t)\partial_tu\partial_tvdt.$$
Indeed, by lemma \ref{l1th2} and our assumptions on $\xb(\tau),$ we can easily obtain that
\be
\int_0^{\xb(\tau)}v^2\sin^{n-3}t\rd t\leq C(n)\int_0^{\xb(\tau)}\sin^{n-2}(t)|\partial_tv|^2\rd t.\label{in3}
\ee
By above inequality we can prove that the space $\mathcal{H}(\xo(\tau))$ is compactly embedded in
$$L^2_{\sin t}((0,\xb(\tau))):=\left\{u:\;(0,\xb(\tau))\rightarrow\mathbb{R}:\;\int_0^{\xb(\tau)}u^2\sin^{n-2}(t)\rd t<\infty\right\}.$$
Thus using standard arguments we can prove that the eigenvalue problem
$$0<\xl(\tau)=\inf_{u\in\mathcal{H}((0,\xb(\tau)))}\frac{\int_0^{\xb(\tau)}\sin^{n-2}(t)\left|\frac{du}{dt}\right|^2\rd t}{\int_0^{\xb(\tau)}u^2\sin^{n-2}(t)\rd t},$$
has a positive minimizer $\xf_1(\tau,t)\in\mathcal{H}(0,\xb(\tau)).$

But,
\begin{align}\nonumber
C(n)\int_0^{\xb(\tau)}\sin^{n-2}(t)|\partial_t\xf_1|^2\rd t&=\int_{\xo}|\nabla\xf_1|^2\rd S,\\
C(n)\int_0^{\xb(\tau)}\sin^{n-2}(t)|u|^2\rd t&=\int_{\xo}|\xf_1|^2\rd S=1,\label{eq1*}
\end{align}

thus $\xf_1\in H_0^1(\xo(\tau))$ and is a weak solution of the eigenvalue problem (\ref{eigen}). Hence  by standard elliptic arguments we can prove that $\xf_1\in L^\infty(\xo(\tau)).$
In addition by our assumption we have that
\be
\sup_{\tau\in\mathbb{R}}\sup_{t\in(0,\xb(\tau))}\left|\xf_1(\tau,t)\right|< C.\label{esteig1}
\ee
By the ODE equation (\ref{eigenode}) and the estimate (\ref{esteig1}), we can write
\be
\xf_1(\tau,t)=\xl\int_t^{\xb(\tau)}\frac{1}{\sin^{n-2}s}\int_0^s\sin^{n-2}(r)\xf_1(\tau,r)\rd r\rd s.\label{repr}
\ee
Thus we have the following estimates
\begin{align}\nonumber
\sup_{\tau\in\mathbb{R}}\sup_{t\in(0,\xb(\tau))}\left|\frac{1}{\sin t}\partial_t\xf_1(\tau,t)\right|&\leq C \sup_{\tau\in\mathbb{R}}\sup_{t\in(0,\xb(\tau))}\left|\xf_1(\tau,t)\right|\\
\sup_{\tau\in\mathbb{R}}\sup_{t\in(0,\xb(\tau))}\left|\partial_t^2\xf_1(\tau,t)\right|&\leq C\sup_{\tau\in\mathbb{R}}\sup_{t\in(0,\xb(\tau))}\left|\xf_1(\tau,t)\right|.\label{esteig2}
\end{align}
Setting now $s=\frac{t}{\xb(\tau)},$ we have that $\xf_1=\xf_1(\tau,s)$ satisfies
\begin{align}
\begin{cases}
\frac{1}{\xb^2(\tau)}\partial_s^2\xf_1(\tau,s) +\frac{(n-2)\cos(\xb(\tau)s)}{\xb(\tau)\sin(\xb(\tau)s)}\partial_s \xf_1(\tau,s)
+\xl(\tau) \xf_1(\tau,s)&=0\qquad\mathrm{in}\;\;(0,1)\\ \nonumber
\xf_1(1)&=0\\
\partial_t\xf_1(0)&=0.
\end{cases}\label{eigenode1}
\end{align}

It is easy to see  that $\lim\limits_{h\rightarrow0}\xf_1(\tau+h,s)=\xf_1(\tau,s)$ in $L^\infty(\mathbb{R}\times(0,1))$. We set
$$u_h(\tau)=\frac{\xf_1(\tau+h,s)-\xf_1(h,s)}{h},\qquad\xf_1(\tau)=\xf_1(\tau,t),$$
then $u_h$ satisfies

\bea\nonumber
&&\frac{1}{\xb^2(\tau+h)}\partial_s^2u_h(\tau) +\frac{(n-2)\cos(\xb(\tau+h)s)}{\xb(\tau+h)\sin(\xb(\tau+h)s)}\partial_s u_h(\tau)
+\xl(\tau+h) u_h(\tau)\\ \nonumber
&=&-\frac{\frac{1}{\xb^2(\tau+h)}-\frac{1}{\xb^2(\tau)}}{h}\partial_s^2\xf_1(\tau)-\frac{\xl(\tau+h)-\xl(\tau)}{h}\xf_1(\tau) \\
&-&(n-2)\frac{\frac{\cos(\xb(\tau+h)s)}{\xb(\tau+h)\sin(\xb(\tau+h)s)}-\frac{\cos(\xb(\tau)s)}{\xb(\tau)\sin(\xb(\tau)s)}}{h}\partial_s \xf_1(\tau)= F_h(\tau,s), \label{pao}
\eea
with $u_h(\tau,1)=\partial_su_h(\tau,0)=0.$
On the other hand  notice that
\bea\nonumber
&&\sup_{\tau\in\mathbb{R}}\left|(n-2)\frac{\frac{\cos(\xb(\tau+h)s)}{\xb(\tau+h)\sin(\xb(\tau+h)s)}
-\frac{\cos(\xb(\tau)s)}{\xb(\tau)\sin(\xb(\tau)s)}}{h}\partial_s\xf_1(\tau,s)\right|\\
&\leq&\sup_{\tau\in\mathbb{R}}\left|(n-2)\left(-\frac{\xb'(\tau)}{\xb^2(\tau)}\cot(\xb(\tau)s)-\frac{s\xb'(\tau)}{\sin^2\xb(\tau)s
}\right)\partial_s \xf_1(\tau,s)\right|\\[3mm]\nonumber
&<&C(n,\inf_{\tau\in\mathbb{R}}\xb(\tau)),\label{claimin1}
\eea
where in the last inequality we have used (\ref{esteig2}) and our assumptions on $\xb.$ Also using our assumption on $\xl$ we have that
\be
\sup_{h\in\mathbb{R}}\sup_{\tau\in\mathbb{R}}F_h(\tau,s)<C(n,\inf_{\tau\in\mathbb{R}}\xb(\tau)).\label{Fh}
\ee
Finally combining  above estimates(\ref{pao})-\eqref{Fh} we have
\be
\lim_{h\rightarrow0}\sup_{\tau\in\mathbb{R}}\int_0^1u_h^2(\tau,s)\sin^{n-2}(\xb(\tau)s)\rd s<C<\infty.\label{claim}
\ee
By (\ref{claim}) we can prove
$$\sup_{\tau\in\mathbb{R}}\sup_{\tau\in\xo(\tau)}|u_h|<C$$
and we have the following representation formula
\bea
\nonumber
\frac{u_h(\tau,s)}{\xb^2(\tau+h)}&=&\xl(\tau+h)\int_s^1\frac{1}{\sin^{n-2}(\xb(\tau+h)\xi)}\int_0^\xi\sin^{n-2}(\xb(\tau+h)r)u_h(\tau,r)drd\xi\\ \nonumber
&-&\int_s^1\frac{1}{\sin^{n-2}(\xb(\tau+h)\xi)}\int_0^\xi\sin^{n-2}(\xb(\tau+h)r)F_h(\tau,r)drd\xi.
\eea
The rest of the proof is standard and we omit it.
\end{proof}

\end{document}